\documentclass[12pt,reqno]{amsart}
\usepackage{amscd}
\usepackage{hyperref}
\usepackage{amsmath,amsfonts,amsthm,amssymb, paralist}
\usepackage{enumitem}

\newcommand{\g}{\geqslant}
\newcommand{\sph}{\mathbb{S}}

\newcommand{\RR}{\mathbb{R}}
\newcommand{\ZZ}{\mathbb{Z}}
\newcommand{\CC}{\mathbb{C}}
\newcommand{\s}{\mathcal{S}}
\newcommand{\NN}{\mathbb{N}}
\newcommand{\p}{\partial}
\newcommand{\q}{\varphi}
\newcommand{\les}{\leqslant}
\newcommand{\lesa}{\lesssim}
\newcommand{\B}{ \dot{{B}}^{\alpha}_{p, q}}
\newcommand{\supp}{\text{supp }\,}
\newcommand{\F}{ \dot{{F}}^{\alpha}_{p, q}}

\newcommand{\Oh}{\mathcal{O}}
\newcommand{\mc}[1]{\mathcal{#1}}

\theoremstyle{plain}
\newtheorem{theorem}{Theorem}[section]
\newtheorem{proposition}[theorem]{Proposition}
\newtheorem{lemma}[theorem]{Lemma}

\newtheorem{corollary}[theorem]{Corollary}

\theoremstyle{definition}
\newtheorem{definition}{Definition}[section]

\theoremstyle{remark}
\newtheorem{remark}{Remark}[section]

\setdefaultenum{(i)}{(a)}{1}{A}
\title[Characterisation of B-L and T-L spaces]{A Characterisation of the Besov-Lipschitz  and Triebel-Lizorkin spaces using Poisson like Kernels}
\author{Huy-Qui Bui and Timothy Candy}
\address{School of Mathematics \& Statistics, University of Canterbury, Private Bag 4800, Christchurch 8140, New Zealand.} \email{huy-qui.bui@canterbury.ac.nz}
\address{Fakult\"{a}t f\"{u}r Mathematik, Universit\"{a}t Bielefeld, Postfach 10 01 31, 33501 Bielefeld, Germany.}
\email{tcandy@math.uni-bielefeld.de}
\date{\today}
\dateposted{}

\begin{document}




\begin{abstract}
We give a complete characterisation of the spaces $\B$ and $\F$ by using a non-smooth kernel satisfying near minimal conditions. The tools used include a Str\"omberg-Torchinsky type estimate \cite{Stromberg1989}  for certain maximal functions and the concept of a distribution of finite growth, inspired by Stein \cite{Stein1993}. In addition, our exposition also makes essential use of a number of refinements of the well-known Calder\'{o}n reproducing formula.
The results are then applied to obtain the characterisation of these spaces via a fractional derivative of the Poisson kernel.
Moreover, our results offer an approach to deal with the calculus modulo polynomials in homogeneous function spaces, a subtle problem raised recently by Triebel \cite{Triebel2015}.
\end{abstract}

\maketitle

\tableofcontents

\section{Introduction and statements of main results} We aim to characterise the Besov-Lipschitz space $\B$ and the Triebel-Lizorkin space $\F$ using a kernel $\psi$ which satisfies near ``minimal'' conditions regarding cancellation, smoothness and decay. In particular, we give a direct characterisation that works in the full scale $\alpha \in \RR$, $0<p, q\les \infty$ in the important case of fractional derivatives of the Poisson kernel (in other words, via harmonic functions). In addition, in the case $\alpha<\frac{n}{p}$, we give a general characterisation of $\dot{B}^\alpha_{p, q}$ and $\dot{F}^\alpha_{p, q}$ as subsets of $\s'$ (rather than the more standard $\s'/\mc{P}$); see Remark \ref{rem - mod polynomials not needed} below.

The Besov-Lipschitz and Triebel-Lizorkin scales of spaces arise in many applications in mathematics. In particular, they are of crucial importance in interpolation theory \cite{Berg1976,Peetre1976,Triebel1992} and contain many well known function spaces in mathematical analysis. For instance, $\dot{F}^0_{p, 2}$ is identified with the real-variable Hardy
space $H^p$ of Fefferman-Stein \cite{FS1972}, while $\dot{F}^0_{\infty, 2}$ is identified with $BMO$, the space of functions of bounded mean oscillation \cite{Frazier1990}.\\

To facilitate the discussion to follow, we begin by recalling the definition of the homogeneous Besov-Lispchitz and Triebel-Lizorkin spaces following Peetre \cite{Peetre1975,Peetre1976} (see also \cite{Triebel1992}). All functions and distributions are defined on the Euclidean space $\RR^n$ unless otherwise stated. Let $\q \in \s$ such that $\supp \widehat{\q} = \{ 1/2 \les |\xi| \les 2 \}$ and for every $\xi \not = 0$
        \begin{equation}\label{eqn - intro sum cond on q} \sum_{j  \in \ZZ} \widehat{\q}(2^{-j} \xi) \widehat{\q}(2^{-j} \xi) = 1. \end{equation}
The function $\q$ is fixed throughout this article. Given a function $\phi:\RR^n \rightarrow \CC$ we let $\phi_j(x) = 2^{jn} \phi(2^j x)$ denote the \emph{dyadic dilation} of $\phi$. Somewhat confusingly, for $t \in \RR$ with $t>0$, we let $ \phi_t(x) = t^{-n} \phi( t^{-1}x )$ denote the standard dilation. It should always be clear from the context the type of dilation we are using.

Let $0< p, q \les \infty$ and $\alpha \in \RR$. The (homogeneous) \emph{Besov-Lipschitz} space $\B$ is then defined as the class of all tempered distributions  modulo polynomials $f \in \s'/\mathcal{P}$ such that
    \begin{equation}\label{eqn - intro defn of B} \| f \|_{\B} = \bigg( \sum_{j \in \ZZ} \Big( 2^{j \alpha} \| \q_j * f \|_{L^p} \Big)^q \bigg)^\frac{1}{q} < \infty.\end{equation}
Similarly, for $0< p, q \les \infty$, $p \not = \infty$, and $\alpha \in \RR$, the (homogeneous) \emph{Triebel-Lizorkin} space $\F$ is defined as the class of all $f \in \s'/\mathcal{P}$ such that
    \begin{equation}\label{eqn - intro defn of F} \| f \|_{\F} = \bigg\| \bigg(\sum_{j \in \ZZ} \Big( 2^{j \alpha} |\q_j * f| \Big)^q \bigg)^\frac{1}{q}\bigg\|_{L^p}  < \infty. \end{equation}
The definition of $\dot{F}^\alpha_{\infty,q}$ is slightly different. The problem is that a naive extension of (\ref{eqn - intro defn of F}) to the case $p=\infty$ leads to spaces which are \emph{not} independent of the choice of kernel, and moreover the expected identification $\dot{F}^0_{\infty, 2}  = BMO$ fails; see the discussion in Section 5 of \cite{Frazier1990}. Instead, following the work of \cite{Frazier1990} we define
    $$ \| f \|_{\dot{F}^\alpha_{\infty, q}} = \sup_{Q} \bigg( \frac{1}{|Q|} \int_Q \sum_{j \g - \ell(Q)} \big( 2^{ j \alpha} |\q_j*f(x)|\big)^q dx \bigg)^\frac{1}{q}$$
with the interpretation that when $q=\infty$,
\begin{equation}\label{Finfty}
 \| f \|_{\dot{F}^\alpha_{\infty, \infty}} = \sup_{Q} \sup_{j \g - \ell(Q)}
 \frac{1}{|Q|} \int_Q  \, 2^{ j \alpha} |\q_j*f(x)|\,dx,
\end{equation}
where the sup is over all dyadic cubes $Q$, and $\ell(Q)= \log_2( \text{ side length of $Q$ })$. The above interpretation for the norm
$\|\cdot\|_{\dot{F}^\alpha_{\infty, \infty}}$ is taken from \cite{Bui2000} where one can also find the embedding $\dot{F}^\alpha_{\infty,q} \subset
\dot{F}^\alpha_{\infty, \infty} = \dot{B}^\alpha_{\infty, \infty}.$

Observe that elements of the quasi Banach spaces $\dot{B}^\alpha_{p, q}$ and $\dot{F}^\alpha_{p, q}$ are equivalence classes of distributions modulo polynomials. However we often make a common (and harmless) abuse of notation by regarding elements of $\dot{B}^\alpha_{p, q}$ and $\dot{F}^\alpha_{p, q}$ as distributions, rather than equivalence classes. \\

The fundamental and central result in the study of the spaces $\dot{B}^\alpha_{p, q}$ and $\dot{F}^\alpha_{p, q}$  is the independence of these function spaces on the choice of the kernel function $\q$. Thus, given a kernel $\psi \in \s$ satisfying certain conditions, if we replace $\q$ in (\ref{eqn - intro defn of B}) and (\ref{eqn - intro defn of F}) with $\psi \in \s$ we have an equivalent norm.  This independence was initially established in the pioneering works of Peetre \cite{Peetre1975,Peetre1976} for all Besov-Lipschitz spaces and for the Triebel-Lizorkin spaces
when $p < \infty$, and the result applied in particular to band-limited kernels.  The method used by Peetre was inspired by the real-variable theory for various maximal functions, developed in the seminal paper \cite{FS1972} by Fefferman and Stein. The result for $\dot{F}^\alpha_{\infty, q}$ was proved later in \cite{Frazier1990} (see also \cite{Bui2000}). After further partial results by Triebel \cite{Triebel1988,Triebel1992},
the essentially optimal conditions, at least in the Schwartz case $\psi \in \s$, were developed in Bui, Paluszy\'nski and Taibleson in \cite{Bui1996}, \cite{Bui1997} and \cite{Bui2000} where it was shown that we have an equivalent norm for any $f\in \s'/\mathcal{P}$, provided that the kernel $\psi \in \s$ satisfies the following:
 \begin{enumerate}[label=(\Roman*)]
    \item (\emph{Vanishing moments}) The kernel $\psi$ has $[\alpha]$ \emph{vanishing moments}, thus
                $$ \int_{\RR^n} x^\kappa \psi(x) dx = 0$$
         for every $|\kappa|\les [\alpha]$, with the understanding that no condition is required when $\alpha<0$.\\

    \item (\emph{Tauberian condition}) The kernel $\psi$ satisfies the  \emph{Tauberian condition}; that is,  for every $\xi \in \mathbb{S}^{n-1}$ there exists $a, b\in \RR$ (depending on $\xi$) with $0<2a\le b$ such that for every $a\le t\le b$
                    $$ |\widehat{\psi}( t \xi) |  >0.$$
    \end{enumerate}
Here, given $\alpha \in \RR$ we let $[\alpha]$ denote the integer part of $\alpha$. Note that the conditions (I) and (II) do not require that the kernel $\psi$ is band-limited. Thus, for example,  it is possible to characterise the Besov-Lipschitz and Triebel-Lizorkin spaces with derivatives of the Gaussian kernel $e^{ - |x|^2}$. \\

In this paper we consider the problem of removing the assumption $\psi \in \s$. Our key goal is to give conditions on the kernel that are simply stated and easily checked, while still being as close as possible to be optimal. Moreover, they should apply in particular to the important case of fractional derivatives of the Poisson kernel $ ( 1 + |x|^2)^{-\frac{n+1}{2}}$.

The first, and somewhat obvious, obstacle in the non-smooth case $\psi \not \in \s$, is that to define the norms (\ref{eqn - intro defn of B}) and (\ref{eqn - intro defn of F}) we need to be able to define the convolution $\psi *f$ for arbitrary $f \in \s'$. This is clearly not possible unless $\widehat{\psi}$ is infinitely differentiable and all its derivatives are slowly increasing. As our main application, the Poisson kernel, does not satisfy the last conditions (its Fourier transform is not differentiable at the origin),
we need to restrict the class of distributions slightly to a natural class of admissible distributions. To this end, inspired by Stein \cite{Stein1993}, we introduce the concept of  \emph{distributions of bounded growth}.

\begin{definition}\label{defn dist of finite growth}[Distributions of growth $\ell$]
We say $f \in \s'$ is a distribution of growth $\ell \g 0$ if for any $\phi \in \s$ we have $\phi * f = \mathcal{O}(|x|^\ell)$ (as $|x| \to \infty$).
\end{definition}

The importance of this definition is that it allows us to make sense of the convolution $\psi*f$ when $\psi \not \in \s$.

\begin{definition}\label{Defn of convolution}
  Assume $f$ is a distribution of growth $\ell$. Then if $(1+|\cdot|)^\ell \psi \in L^1$ we define the convolution $\psi * f \in \s'$ as
            $$ \psi*f(\phi) = \int_{\RR^n} \psi(x) (\widetilde{\phi}*f)(x) dx,
 \quad \phi\in \s,$$
where $\widetilde{\phi}(x)=\phi(-x)$. This definition coincides with the pointwise definition for $\psi*f$ when $f=\Oh(|x|^\ell)$ is locally integrable.
\end{definition}

We note that Stein used the concept of a bounded distribution (the case $\ell = 0$ in Definition \ref{defn dist of finite growth}) to characterise the Hardy spaces $H^p$ using the Poisson kernel (see \cite{Stein1993, FS1972}).\\

Before proceeding to state the main theorem we prove, we discuss the key conditions we require on our kernel. To this end, we take parameters $\Lambda \g 0$ and $m, r \in \RR$, and suppose $\psi \in L^1(\RR^n)$.

\begin{enumerate}
  \item[(\textbf{C1})] (\emph{Cancelation}) Let $\widehat{\psi} \in C^{n+1+[\Lambda]}(\RR^n \setminus \{0\})$ such that for every $|\kappa| \les n+1+[\Lambda]$ we have
                $$ \p^\kappa \widehat{\psi} = \mathcal{O}( |\xi|^{r-|\kappa|}) \qquad \text{ as $|\xi| \rightarrow 0$.} $$

  \item[(\textbf{C2})] (\emph{Tauberian condition}) The kernel $\psi$ satisfies the Tauberian condition (as in (II) above).

  \item[(\textbf{C3})] (\emph{Smoothness}) Take $\widehat{\psi} \in C^{n+1+[\Lambda]}(\RR^n\setminus \{ 0\})$ such that for every $|\kappa| \les n+1+[\Lambda]$ we have
                $$ \p^\kappa \widehat{\psi} = \mathcal{O}( |\xi|^{-n-m}) \qquad \text{ as $|\xi| \rightarrow \infty$.} $$
\end{enumerate}

The parameters $\Lambda$ and $m$ roughly correspond to the decay and smoothness we require on a component of our kernel $\psi$. More precisely if $\phi \in \s$ has Fourier support away from the origin, then $\widehat{\psi} \in C^{n+1+[\Lambda]}(\RR^n \setminus\{0\})$ implies that $\psi *\phi(x) = \mathcal{O}(|x|^{-n-1-[\Lambda]})$ as $|x| \rightarrow \infty$. Thus larger $\Lambda$ requires more decay on the part of $\psi$ with Fourier support away from the origin. Similarly, if (\textbf{C3}) holds, then $\psi \in C^{[m]}(\RR^n)$. Thus the larger we take $m$, the smoother the kernel $\psi$ is required to be.

On the other hand, the parameter $r$ and the cancellation condition (\textbf{C1}) are closely related to the vanishing moments condition (I). More precisely, if $\psi \in \s$ then it is easy to check that $\psi$ has $[\alpha]$ vanishing moments (i.e. (I) holds), if and only if (\textbf{C1}) holds with $\alpha < r \les [\alpha]+1$. Of course if $\psi \not \in \s$ then the relationship between (I) and (\textbf{C1}) is somewhat complicated, but roughly speaking (I) requires more \emph{spatial} decay, while (\textbf{C1}) requires more regularity of $\widehat{\psi}$ near the origin. It is worth pointing out that it is possible to prove the characterisations below with (\textbf{C1}) replaced with (I), but this requires more decay on the kernel $\psi$ and leads to less optimal conditions. Instead we have chosen to use conditions on the Fourier transform of $\psi$, as firstly this matches up very well with our intended application to the Poisson kernel, and secondly, in the authors' opinion the conditions (\textbf{C1}), (\textbf{C2}), and (\textbf{C3}) form an acceptable balance between the sharpness of our result, and the simplicity of its statement. \\

 It is natural to split the characterisation results into two theorems:  ``Necessary Conditions'' and ``Sufficient Conditions''.  This is due to fact that, as noted in \cite{Bui1996}, each theorem requires a slightly different set of assumptions. The essential assumption for the former is the cancellation property of the kernel, expressed by the condition  (\textbf{C1}), while for the latter the Tauberian condition (II) stated earlier in the introduction is critical. Other conditions, such as the decay at infinity of the kernel in the frequency domain expressed by the smoothness condition (\textbf{C3}), are needed to define the convolution with distributions of finite growth.\\

In the necessary direction, the statement of our result is expressed using a maximal function introduced in the work of Peetre \cite{Peetre1975}. Given a kernel $\psi$, if $f\in \s'$ is such that each $\psi_j*f$ is a function, one defines the \emph{Peetre maximal function} by
\begin{equation}\label{eqn - defn Peetre max func}
\psi_j^*f(x) = \psi^*_{j,\lambda}f(x) = \sup_{y\in \RR^n}\frac{|\psi_j*f(x-y)|}{(1+2^j|y|)^\lambda}\;,
\quad x\in \RR^n,
\end{equation}
where $\lambda>n/p$ in the Besov-Lipschitz case and $\lambda>\max\{n/p,\,n/q\}$ in the Triebel-Lizorkin case (with $\lambda>n$ for
the space $\dot{F}^\alpha_{\infty, \infty}$). Unless otherwise stated, the number $\lambda$  satisfies these conditions throughout this work.

In the rest of the paper we write $A \lesa B$ when there exists a positive constant $C$ such that $A \les CB$, where $C$ may depend on the parameters such as $n,\alpha,p,q,...$ but usually not on the variable quantities such as the distribution $f$. When $j,k\in \ZZ$ we write $j\lesa k$ to mean that $j\les k+c$ for some $c\in \ZZ$ independent of $j$ and $k$.\\

We can now state our main results. We start with the necessary direction.

\begin{theorem}\label{thm - main thm necessary cond}
Let $\alpha \in \RR$ and $0<p, q \les \infty$. Let $ \ell\g0$ with  $\ell > \alpha - \frac{n}{p}$. Assume $(1+|\cdot|)^\ell \psi \in L^1$ and that $\psi$ satisfies the cancellation condition (\textbf{C1}) and smoothness condition (\textbf{C3}) for parameters $\Lambda \g 0$, $r>\alpha$, and $m > \Lambda - \alpha$.
\begin{enumerate}
\item
Let $f \in \B$ and $\Lambda = \frac{n}{p}$. Then there exists a polynomial $\rho$ such that $f-\rho$ is a distribution of growth $\ell$ and we have the inequalities
\begin{equation}\label{necessary besov}
\Big( \sum_{j \in \ZZ} \big( 2^{j\alpha} \| \psi^*_j (f-\rho)\|_{L^p} \big)^q \Big)^{\frac{1}{q} } \lesa \| f \|_{\B},
\end{equation}
and for any $\phi \in \s$
\begin{equation}\label{necessary hardy besov}
\Big( \sum_{j \in \ZZ} \Big( 2^{j\alpha} \Big\| \sup_{t>0}|\phi_t*\psi_j * (f-\rho)|\Big\|_{L^p} \Big)^q \Big)^{\frac{1}{q} } \lesa \| f \|_{\B}.
\end{equation}

\item
Similarly if $f\in \F$ and we let $\Lambda= \max\{ \frac{n}{p}, \frac{n}{q} \}$ (with $\Lambda=n$ when $p=q=\infty$),  then there exists a polynomial $\rho$ such that $f-\rho$ is a distribution of growth $\ell$ and if $p<\infty$
        \begin{equation}\label{necessary triebel}
                \Big\| \Big( \sum_{j \in \ZZ} \big( 2^{j\alpha}  \psi^*_j (f-\rho)  \big)^q \Big)^{\frac{1}{q} } \Big\|_{L^p} \lesa \| f \|_{\F},
        \end{equation}
and for $p=\infty$
        \begin{equation}\label{necessary triebel p=infty}
           \sup_Q \bigg(\frac{1}{|Q|} \int_Q \sum_{j \g - \ell(Q)} \big( 2^{ j \alpha} \psi^*_j(f-\rho)(x) \big)^q dx\bigg)^\frac{1}{q} \lesa \| f\|_{\dot{F}^\alpha_{\infty, q}}.
        \end{equation}
Furthermore, if $\phi \in \s$, we have for $p<\infty$
        \begin{equation}\label{necessary hardy triebel}
             \Big\| \Big( \sum_{j \in \ZZ} \big( 2^{j\alpha}  \sup_{t>0} |\phi_t*\psi_j* (f-\rho) | \big)^q \Big)^{\frac{1}{q} } \Big\|_{L^p} \lesa \| f \|_{\F}
        \end{equation}
and $p=\infty$
        \begin{equation}\label{necessary hardy triebel p=infty}
           \sup_Q \bigg( \frac{1}{|Q|} \int_Q \sum_{j \g - \ell(Q)} \big( 2^{ j \alpha} \sup_{t>0} |\phi_t*\psi_j*(f-\rho)(x)| \big)^q dx \bigg)^\frac{1}{q}\lesa \| f\|_{\dot{F}^\alpha_{\infty, q}}.
        \end{equation}
\end{enumerate}
\end{theorem}

\begin{remark}
A few remarks are in order.

(i) When $q=\infty$, the inequality
(\ref{necessary hardy triebel p=infty}) is interpreted similarly to the definition of the $\dot{F}^\alpha_{\infty, \infty}$-norm (\ref{Finfty}); i.e.,
$$
\sup_Q \sup_{j \g - \ell(Q)} \frac{1}{|Q|}\int_Q \big(2^{ j \alpha} \sup_{t>0} |\phi_t*\psi_j*(f-\rho)(x)|\big)  dx \lesa \| f\|_{\dot{F}^\alpha_{\infty, \infty}}.
$$
We adopt this interpretation hereafter in all the theorems and proofs.

(ii) The assumptions on the kernel $\psi$ ensure that each convolution $\psi_j*(f-\rho)$ is well-defined and moreover is a continuous function (see Theorem \ref{thm - calderon and the pointwise defn of conv}). This is a consequence of two key steps. The first is to show that if $f \in \B$ (or $f \in \F$), then there is a polynomial $\rho$ such that $f-\rho$ is in fact a distribution a growth $\ell > \alpha - \frac{n}{p}$ (see Theorem \ref{thm - classical calderon on B} and Corollary \ref{cor - growth bound on B and F}). Thus we can define the convolution $\psi*(f-\rho)$ as a \emph{distribution} via Definition \ref{defn dist of finite growth}. The second step is to use a version of the Calder\'{o}n reproducing formula to deduce that the distribution $\psi * (f - \rho)$ is in fact a continuous function (see Theorem \ref{thm - calderon and the pointwise defn of conv}). It is important to note that both of these steps rely crucially on the fact that we assume $f \in \B$ (or $f \in \F$), and thus have some control over the growth \emph{and} smoothness of the distribution $f$.

 (iii) Given $\phi \in \s$ with $\int \phi \not = 0$ and $0<p\les\infty$, the characterisation of the real-variable Hardy space $H^p$, defined by Fefferman and Stein in \cite{FS1972}, gives
             $$ \Big\|\sup_{t>0}|\phi_t*g|\Big\|_{L^p} \approx \|g\|_{H^p}.$$
(Note that when $1<p\les \infty$, $H^p=L^p$ with equivalent norms.)
 Thus it follows from (\ref{necessary hardy besov}) that
\begin{equation}\label{necessary hardy besov2}
\Big( \sum_{j \in \ZZ} \Big( 2^{j\alpha} \Big\|\psi_j * (f-\rho)\Big\|_{H^p} \Big)^q \Big)^{\frac{1}{q} } \lesa \| f \|_{\B}.
\end{equation}

(iv) Since $|\psi_j*g|$ is clearly dominated pointwise by $\psi_j^*g$ we may replace (\ref{necessary besov}) with
\begin{equation}\label{necessary besov2}
\Big( \sum_{j \in \ZZ} \big( 2^{j\alpha} \| \psi_j* (f-\rho)\|_{L^p} \big)^q \Big)^{\frac{1}{q} } \lesa \| f \|_{\B}.
\end{equation}
Similarly we may replace the Peetre maximal function $\psi^*_j (f-\rho)$ in (\ref{necessary triebel}) and (\ref{necessary triebel p=infty}) with the standard convolution $|\psi_j*(f-\rho)|$.
\end{remark}

We next consider the converse to the above theorem; that is, to find sufficient conditions on the kernel $\psi$ and the distribution $f$ such that the reverse inequalities to those in Theorem \ref{thm - main thm necessary cond} holds.
We emphasise that the results in this sufficient part are the main contribution of this paper.
As soon as the assumption $\psi\in \s$ is dropped, one immediately runs into the difficulty of defining the convolution $\psi_j*f$ when $f\in \s'$. The situation is different from the necessary result in Theorem \ref{thm - main thm necessary cond} where we already knew that $f \in \B$ or $f\in \F$, and therefore the convolution given in Definition \ref{Defn of convolution} can be seen to be a continuous bounded function via what is essentially a duality argument (see Theorem \ref{thm - calderon and the pointwise defn of conv}).
However, if $f$ is a distribution of growth $\ell\g 0$, then we have seen  it is possible to define $\psi_j*f$ as a distribution under rather mild condition on $\psi$ (see Definition \ref{Defn of convolution}). Our first result in the sufficient direction makes use of this observation.

\begin{theorem}\label{thm - sufficient hardy}
Let $\alpha \in \RR$ and $0<p, q \les \infty$. Assume $f$ is a distribution of growth $\ell \g 0$. Suppose $(1+|\cdot|)^\ell \psi \in L^1$ satisfies the Tauberian condition (\textbf{C2}) and there exists $m \in \RR$ such that the smoothness condition (\textbf{C3}) holds for $\Lambda \g 0$.
\begin{enumerate}
\item Let $\Lambda=\max\{\ell, \frac{n}{p}\}$. Then

\begin{equation}\label{sufficient - hardy besov}
\| f\|_{\B} \lesa \Big( \sum_{j \in \ZZ} \big( 2^{j \alpha} \| \psi_j*f\|_{H^p}\big)^q \Big)^{\frac{1}{q}}.
\end{equation}
\item
Let $\Lambda=\max\{\ell, \frac{n}{p}, \frac{n}{q}\}$ and $\phi \in \s$ with $\int \phi(x)dx \neq 0$. If $p<\infty$ then
\begin{equation}\label{sufficient - hardy triebel}
 \| f\|_{\F} \lesa \Big\| \Big( \sum_{j \in \ZZ} \big( 2^{j\alpha} \sup_{t>0} |\phi_t*\psi_j*f|\big)^q \Big)^{\frac{1}{q}} \Big\|_{L^p},
\end{equation}
and in the case $p=\infty$
\begin{equation}\label{sufficient - hardy p infty triebel}
  \| f \|_{\dot{F}^\alpha_{\infty, q}} \lesa \sup_{Q} \bigg( \frac{1}{|Q|} \int_Q \sum_{j \g - \ell(Q)} \big( 2^{j \alpha} \sup_{t>0}|\phi_t*\psi_j*f(x)|\big)^q dx \bigg)^\frac{1}{q}
\end{equation}
(with $\Lambda=\max\{n, \ell\}$ when $p=q=\infty$).
\end{enumerate}
\end{theorem}

\begin{remark}\leavevmode

(i) Observe that we are free to choose the smoothness parameter $m$, thus an alternative way to state the smoothness condition on $\psi$, is that we simply require $\p^\kappa \widehat{\psi}$ to be slowly increasing for $|\kappa| \les  n + 1 + [\Lambda]$.

(ii) Theorem \ref{thm - sufficient hardy} together with Theorem \ref{thm - main thm necessary cond} give the following complete characterisation of $\B$. Let $\ell> \alpha - \frac{n}{p}$, $r> \alpha$, $m > \frac{n}{p} - \alpha$, and $\Lambda= \max\{\ell, \frac{n}{p}\}$. If the kernel $(1+|\cdot|)^\ell \psi \in L^1$ satisfies the conditions (\textbf{C1}), (\textbf{C2}), and (\textbf{C3}), then $ f \in \B$  if and only if $f$ is a distribution of growth $\ell$ and $\Big( \sum_{j \in \ZZ} \big( 2^{j \alpha} \| \psi_j*f\|_{H^p}\big)^q \Big)^{\frac{1}{q}}< \infty$. A similar comment applies in the Triebel-Lizorkin case.
\end{remark}

If we want a version of Theorem \ref{thm - sufficient hardy} with $H^p$ replaced with $L^p$, we need to assume more on our kernel $\psi$ to ensure that each $\psi_j *f$ is a measurable function (as apposed to just an element of $\s'$). It is worth noting that, under the assumptions of Theorem \ref{thm - sufficient hardy}, if we assume the right hand side of (\ref{sufficient - hardy besov}) is finite, then since $H^p = L^p$ for $p>1$, we may freely replace the $H^p$ norm with the $L^p$ norm in (\ref{sufficient - hardy besov}). Thus at first glance, it may appear that the conditions on $\psi$ in Theorem \ref{thm - sufficient hardy} are sufficient to also deduce an $L^p$ version of (\ref{sufficient - hardy besov}).

However this is slightly misleading, as we may only replace $H^p$ with $L^p$ after making the a priori assumption that the right hand side of (\ref{sufficient - hardy besov}) is finite. Without this finiteness assumption, it is not possible to ensure that the distribution $\psi_j*f$ is in a fact a function. Thus in general, under the assumptions on $\psi$ in Theorem \ref{thm - sufficient hardy}, the norm $\| \psi_j*f \|_{L^p}$ is not defined. Consequently, if our goal is to prove a direct characterisation without any auxiliary assumptions on the distribution $f$, to ensure that $\psi_j*f$ is a measurable function, we need to make further assumptions on our kernel $\psi$. One such condition is found in the next theorem.

\begin{theorem}\label{thm - sufficient with rapidly decreasing}
Let $\alpha \in \RR$ and $0<p, q \les \infty$. Assume $f$ is a distribution of growth $\ell \g 0$. Suppose $(1+|\cdot|)^\ell \psi \in L^1$ satisfies the Tauberian condition (\textbf{C2}) and that for every $m \in \RR$, the smoothness condition (\textbf{C3}) holds with $\Lambda \g \ell$.   Then for every $j\in \ZZ$ the convolution $\psi_j*f$ is a continuous function. Moreover, if $\Lambda= \max\{\ell, \frac{n}{p}\}$ then
\begin{equation}\label{sufficient - lebesgue besov}
\| f\|_{\B} \lesa \Big( \sum_{j \in \ZZ} \big( 2^{j \alpha} \| \psi_j *f\|_{L^p} \big)^q \Big)^{\frac{1}{q}}.
\end{equation}
Similarly, if $\Lambda=\max\{ \ell, \frac{n}{p} , \frac{n}{q} \}$ and $p<\infty$ then
\begin{equation}\label{sufficient - lebesgue triebel}
\| f\|_{\F} \lesa \Big\| \Big( \sum_{j \in \ZZ} \big( 2^{j\alpha} |\psi_j*f|\big)^q \Big)^{\frac{1}{q}} \Big\|_{L^p}
\end{equation}
and in the case $p=\infty$
\begin{equation}\label{sufficient - p infty triebel}
  \| f \|_{\dot{F}^\alpha_{\infty, q}} \lesa \sup_{Q} \bigg( \frac{1}{|Q|} \int_Q \sum_{j \g - \ell(Q)} \big( 2^{j \alpha} |\psi_j*f(x)|\big)^q dx \bigg)^\frac{1}{q}
\end{equation}
(with $\Lambda=\max\{n, \ell\}$ when $p=q=\infty$).

\end{theorem}

\begin{remark} (i) It is possible to reduce the smoothness assumption slightly; see Theorem \ref{thm - sufficient Lp assuming maximal func finite} and the proof of Theorem \ref{thm - sufficient with rapidly decreasing} in Section \ref{sec - characterisation thms}. In particular, the smoothness condition (\textbf{C3}) can be replaced with the marginally weaker assumption that $\p^\kappa \widehat{\psi}$ is rapidly decreasing for $|\kappa| \les \max\{ n+1+[\ell], [\Lambda]\}$. This is a fairly strong condition on the kernel $\psi$, but a condition of this type seems to be necessary in order for the convolution $\psi * f$ to have a pointwise definition for every $f \in \s'$ of growth $\ell$, see Remark \ref{remark sharpness of rapidly decreasing assumption}  below. On the other hand, if we instead make further assumptions on the distribution $f$, then it is possible to define the convolution $\psi*f$ without the rapidly decreasing assumption on $\p^\kappa \widehat{\psi}$;  see Theorem \ref{thm - sufficient with f slowly increasing} for a result in this direction.

(ii)  In the above two theorems, we have restricted the class of distributions to those of growth $\ell$. This condition is natural in light of Theorem \ref{thm - main thm necessary cond} where it was shown that all elements of $\B$ and $\F$ are, perhaps modulo a polynomial, distributions of growth $ \alpha - \frac{n}{p} + \epsilon$ for every $\epsilon>0$. Thus we do not lose anything by only considering distributions of some finite growth. Observe also that by making $\ell$ smaller, we weaken the condition on $\psi$, but unfortunately require a \emph{stronger}  growth condition on $f$. A good choice for $\ell$, which is suggested by the necessary results, is to take $\ell>\alpha-n/p$.
\end{remark}

Theorem \ref{thm - main thm necessary cond},  Theorem \ref{thm - sufficient hardy} and Theorem \ref{thm - sufficient with rapidly decreasing} provide necessary and sufficient conditions for a distribution to be in the Besov-Lipschitz space or in the Triebel-Lizorkin space.  In other words, these theorems provide the characterisations of the function spaces under study.  \\

We now come to our main application of the previous results. Namely we give a characterisation of $\B$ and $\F$ via fractional derivatives of the Poisson kernel. Thus we consider the case $\widehat{\psi}(\xi)=|\xi|^\beta e^{-|\xi|}$; that is, $\psi = (-\Delta)^{\beta/2}P$, and $P$  is the Poisson kernel on $\RR^n$,
$$ P(x) = \frac{c_n}{(1+|x|^2)^{(n+1)/2}}\,.$$
 Note that the Poisson kernel case is one of the main motivations for this work.

\begin{theorem}\label{Poisson characterisation}
Let $\alpha \in \RR$, $0<p, q \les \infty$. Let $\beta \g 0$, $\beta>\alpha$, and define $\widehat{\psi}(\xi) = |\xi|^\beta e^{-|\xi|}$. Let $\ell \g 0$ such that
    $$ \alpha - \frac{n}{p} < \ell < \begin{cases} \beta + 1 \qquad &\frac{\beta}{2} \in \NN \\
                                                    \beta &\frac{\beta}{2} \not \in \NN. \end{cases}$$
Assume that $f$ is a distribution of growth  $\ell$. Then the convolution $\psi_j*f$ is a continuous function, and there exists a polynomial $\rho$ of degree at most $[\ell]$ such that
         $$\Big( \sum_{j \in \ZZ} \big( 2^{j \alpha} \| \psi_j^*(f-\rho)\|_{L^p} \big)^q \Big)^{\frac{1}{q}} \lesa \| f\|_{\B} \lesa \Big( \sum_{j \in \ZZ} \big( 2^{j \alpha} \| \psi_j *f\|_{L^p} \big)^q \Big)^{\frac{1}{q}},$$
          $$\Big\| \Big( \sum_{j \in \ZZ} \big( 2^{j\alpha} \psi_j^*(f-\rho)\big)^q \Big)^{\frac{1}{q}} \Big\|_{L^p}\lesa \| f\|_{\F} \lesa \Big\| \Big( \sum_{j \in \ZZ} \big( 2^{j\alpha} |\psi_j*f|\big)^q \Big)^{\frac{1}{q}} \Big\|_{L^p}, \quad p<\infty, $$
and
\begin{align*}
\sup_Q \bigg(\frac{1}{|Q|} \int_Q \sum_{j \g - \ell(Q)} \big( 2^{ j \alpha} &\psi^*_j(f-\rho)(x) \big)^q dx\bigg)^\frac{1}{q} \\
&\lesa
\| f \|_{\dot{F}^\alpha_{\infty, q}}
\lesa \sup_{Q} \bigg( \frac{1}{|Q|} \int_Q \sum_{j \g - \ell(Q)} \big( 2^{j \alpha} |\psi_j*f(x)|\big)^q dx \bigg)^\frac{1}{q}.
\end{align*}
Moreover, in the case $\ell < \beta$, we may take $\rho = 0$.
\end{theorem}
\begin{proof} 
It is obvious that $\psi$ satisfies all the assumptions in Theorem \ref{thm - main thm necessary cond},  Theorem \ref{thm - sufficient hardy} and Theorem \ref{thm - sufficient with rapidly decreasing}, except possibly the integrability condition $(1+|\cdot|)^\ell \psi \in L^1$. But this follows readily from Corollary \ref{Poisson kernel - integrability} in the case $\frac{\beta}{2} \not \in \NN$, and in the case $\frac{\beta}{2} \in \NN$ by observing that $\psi = (-\Delta)^{\frac{\beta}{2}} P$.  Thus we obtain the required inequalities, up to perhaps a polynomial factor $\rho$ in the case of the left hand estimates. However as $f$ and $f-\rho$ are of growth $\ell$, we see that $\rho$ must be of degree at most $[\ell]$. The final conclusion follows by noting that if $\ell< \beta$, then Corollary \ref{cor - fourier decay at origin implies spatial decay} implies that $\psi$ has $[\ell]$ vanishing moments. Therefore $\psi*(f-\rho) = \psi*f$ and hence we may take $\rho  =0$.
\end{proof}

\begin{remark}(i) A similar argument shows that Theorem \ref{Poisson characterisation} holds when
$\widehat{\psi} = |\cdot|^\beta \widehat{\phi}$, where $\phi \in \s$ satisfies the Tauberian condition. The proof is similar to the Poisson kernel case.

(ii) When $\beta=m\in \NN$ in the above theorem, one has
$\psi_t*f = \big(\big(\p/\p t)^mP_t)*f$ (= $\big(\p/\p t\big)^m(P_t*f)$ if $P_t*f$ is defined). This case is historically important as the Poisson kernel was a principal tool used in the classical study of function spaces, in which the mean-value property of the harmonic function $P_t*f$ is crucial. In fact, the sufficient result for the Poisson kernel case (the right-hand side inequalities in the above theorem) has only been proved in the literature using the mean-value property (see \cite{Bui1984}, \cite{Triebel1988}). Moreover, the question of defining the convolution $\psi_t*f$ was not fully elaborated in these works. Also Theorem \ref{Poisson characterisation} for non-integer $\beta$ appears to be new.
\end{remark}

\begin{remark}\label{rem - mod polynomials not needed} This remark relates to a question and some results in the recent manuscript \cite{Triebel2015}.  The first-named author is grateful to Professor Hans Triebel for sending him a copy of this work.

 (i) Fix $0<p, q \les \infty$, $\alpha \in \RR$,  and $\ell \g 0$ such that $\alpha - \frac{n}{p} < \ell < \max([\alpha - \frac{n}{p}],0)+1$. Let $\psi \in L^1$ be a kernel satisfying the conditions of Theorem \ref{thm - main thm necessary cond} and Theorem \ref{thm - sufficient with rapidly decreasing}. Given a distribution $f \in \s'$ of growth $\ell$, we define
    $$ \| f \|_{\dot{B}^\alpha_{p, q}(\psi)} =  \left( \sum_{j \in \ZZ} \left( 2^{j\alpha} \| \psi_j * f \|_{L^p} \right)^q \right)^\frac{1}{q}$$
and take
    $$ \dot{B}^\alpha_{p, q}(\psi) = \big\{ f \in \s' \, \big| \, f \text{ of growth $\ell$ and $\| f \|_{\dot{B}^\alpha_{p, q}(\psi)}<\infty$}\big\}.$$
Theorem \ref{thm - sufficient with rapidly decreasing} shows that $\| f \|_{\dot{B}^\alpha_{p, q}(\psi)}$ is well-defined.  Moreover, together with an application of Theorem \ref{thm - main thm necessary cond}, the (quasi) Banach space $\dot{B}^\alpha_{p, q}(\psi)/( \dot{B}^\alpha_{p, q}(\psi) \cap \mc{P}) $ is isomorphic to $\dot{B}^\alpha_{p, q}$ with equivalent norms (here $\mc{P}$ denote the set of all polynomials). Thus we have a complete characterisation of $\dot{B}^\alpha_{p, q}$, provided we consider elements of $\dot{B}^\alpha_{p, q}(\psi)$ modulo polynomials in $\dot{B}^\alpha_{p, q}(\psi) \cap \mc{P}$. Because of the restriction on the growth of distributions in $ \dot{B}^\alpha_{p, q}(\psi)$,
we need only consider polynomials of degree at most $\max([\alpha-\frac{n}{p}],0)$. Note that, although we restrict our discussion above to the Besov-Lipschitz spaces, appropriate versions hold for the Triebel-Lizorkin spaces.

(ii) In light of Theorem \ref{thm - sufficient hardy}, we may replace the $L^p$ norm in the definition of $\| \cdot \|_{\dot{B}^\alpha_{p, q}(\psi)}$ by the $H^p$ norm (and thus weaken the conditions on $\psi$) and get an equivalent statement to (i). Similarly,  we may replace $\psi_j*f$ with the maximal function $\psi^*_j f$. Moreover, again via Theorem \ref{thm - main thm necessary cond}, Theorem \ref{thm - sufficient hardy}, and Theorem \ref{thm - sufficient with rapidly decreasing}, clearly an equivalent statement is also true for the Triebel-Lizorkin spaces $\dot{F}^\alpha_{p, q}$.

(iii) The fact that the spaces $\dot{B}^\alpha_{p, q}$ (and $\dot{B}^\alpha_{p, q}(\psi)$) are only Banach spaces when considered modulo polynomials can be somewhat inconvenient. However, in the case $\alpha<\frac{n}{p}$, it is possible to remove this ambiguity. One way to do this is as follows. Let
    $$\mc{Z} = \Big\{ f \in \s' \, \Big| \, f = \sum_{j \in \ZZ} \q_j * \q_j *f \,
\; \text{in $\s'$} \Big\}. $$
Thus $\mc{Z}$ is the collection of all distributions for which the Calder\'{o}n reproducing formula holds in $\s'$. (For general distributions, the Calder\'{o}n reproducing formula only holds in $\s'/\mc{P}$; see for instance Theorem \ref{thm - classical calderon on s'}  and Theorem \ref{thm - classical calderon on B} below.) Consider the set $\dot{B}^\alpha_{p, q}(\psi) \cap \mc{Z}  \subset \s'$ (here elements \emph{are not} considered modulo polynomials) and define the map $\Phi: \dot{B}^\alpha_{p, q}(\psi) \cap \mc{Z} \rightarrow \dot{B}^\alpha_{p, q}$ by $\Phi(f) = f + \mc{P}$ (i.e. $\Phi$ maps each distribution in $\dot{B}^\alpha_{p, q}(\psi) \cap \mc{Z}$ to its equivalence class in $\dot{B}^\alpha_{p, q}$). Then \begin{enumerate}
  \item[(a)] $\dot{B}^\alpha_{p, q}(\psi) \cap \mc{Z}$ is a Banach space (as a subset of $\s'$), and the map $\Phi$ is injective with
                        $$ \| f \|_{\dot{B}^\alpha_{p, q}(\psi)} \lesa \| \Phi(f) \|_{\dot{B}^\alpha_{p, q}} \lesa \| f \|_{\dot{B}^\alpha_{p, q}(\psi)};$$

  \item[(b)] if $\alpha<\frac{n}{p}$ then, in addition, the map $\Phi$ is a bijection.
\end{enumerate}
As in (i), the conclusion (a) follows immediately from Theorem \ref{thm - main thm necessary cond} and Theorem \ref{thm - sufficient with rapidly decreasing}. The key point in (b), is that given any $f \in \dot{B}^\alpha_{p, q}$, there exists $g \in f + \mc{P}$ such that $g \in \mc{Z}$  (see Theorem \ref{thm - classical calderon on B} below). Thus in conclusion, if one wishes to consider $\dot{B}^\alpha_{p, q}$ \emph{without} working modulo polynomials, a natural choice is the Banach space $\dot{B}^\alpha_{p, q}(\psi) \cap \mc{Z}$, and this space is equivalent to $\dot{B}^\alpha_{p, q}$ provided $\alpha < \frac{n}{p}$. As in (i) and (ii), a similar argument applies in the case of $\dot{F}^\alpha_{p, q}$.

(iv) Let $\widehat{\psi}(\xi) = e^{-|\xi|}$ denote the Poisson kernel. If $\alpha <0$, we have $ \dot{B}^\alpha_{p, q}(\psi) = \dot{B}^\alpha_{p, q}(\psi) \cap \mc{Z}$ and thus we can use the norm
        \begin{equation}\label{eqn - psi norm besov} \left( \sum_{j \in \ZZ} \left( \| \psi_j * f \|_{L^p}\right)^q \right)^{\frac{1}{q}}\end{equation}
to define $\dot{B}^\alpha_{p, q}$ \emph{without} needing to work modulo polynomials (after applying Theorem \ref{Poisson characterisation}, this just reduces to showing that if $c \in \RR$ and $\| c \|_{\dot{B}^\alpha_{p, q}(\psi)}<\infty$, then $c=0$). More generally, a similar comment applies whenever $\alpha<0$ and we take any kernel with $\int \psi \not = 0$ (provided of course that $\psi$ satisfies the conditions in Theorem \ref{thm - main thm necessary cond} and Theorem \ref{thm - sufficient with rapidly decreasing}). In the case $\alpha \g 0$, this approach does not work as the kernel $\psi$ is required to have some vanishing moments, thus (\ref{eqn - psi norm besov}) being finite does \emph{not} imply that $f \in \mc{Z}$. As previously, an identical argument can also be used in the case of  $\dot{F}^\alpha_{p, q}$.

(v) This, rather lengthy, remark should be compared with the results in the recent work of Triebel \cite{Triebel2015} where, among other things, the problem of defining homogeneous function spaces in the framework of $\s$ and $\s'$ (as opposed to $\s'/\mc{P}$) is considered. In particular,  \cite{Triebel2015} contains a special case of (iv) with $\widehat{\psi} = e^{-|\xi|^2}$ being the Gaussian kernel, and moreover, in the case $\max\{ \frac{n}{p} - n, 0\}<\alpha < \frac{n}{p}$, the spaces $\dot{B}^\alpha_{p, q}(\psi) \cap L^r$ (with the additional restriction $q\les r$) and $\dot{F}^\alpha_{p, q}(\psi) \cap L^r$ with $\alpha = n ( \frac{1}{p} - \frac{1}{r})$, are introduced (with $\psi$ being derivatives of the Gaussian kernel) as suitable substitutes for $\dot{B}^\alpha_{p, q}$ and $\dot{F}^\alpha_{p, q}$. As in (iii), the idea being that one can work \emph{directly} in the framework $\s$ and $\s'$ without needing to worry about $\s'/\mc{P}$. An advantage of (iii) over the results in \cite{Triebel2015} is that one can deal with the whole range $\alpha < \frac{n}{p}$, $0<p, q \les \infty$ for both $\dot{B}^\alpha_{p, q}$ \emph{and} $\dot{F}^\alpha_{p, q}$. Moreover,  when $\alpha \g n/p$,
the results in part (i) of our remark offer an approach to study these function spaces modulo polynomials of degree at most $[\alpha - n/p]$.

\end{remark}
While the general outline of our arguments follows the original works \cite{Bui1996,Bui1997,Bui2000} and, in the necessary part, also the pioneering paper \cite{Peetre1975} by Peetre, the non-smooth assumption on the kernel $\psi$  and its the Fourier transform requires not only substantial technical modifications, but also the introduction of a new concept, the distributions of finite growth. We also benefit from the thesis \cite{Candy2008} where some partial results are obtained.
Of independent interest is our extension of the Calder\'on reproducing formula and the Str\"omberg-Torchinsky estimate in \cite{Stromberg1989} to the non-smooth case. These could be useful in other research in harmonic analysis of function spaces. \\

The plan for the rest of the paper is as follows. In Section \ref{sec - Prelim results} we prove a number of estimates that are used frequently  throughout the paper, in particular we state a growth estimate on elements of $\dot{B}^\alpha_{\infty, \infty}$. Section \ref{sec - pointwise defn of conv} is devoted to the problem of the pointwise definition of the convolution. Section \ref{sec - maximal ineq} is the main part of the paper, where the necessary tools are developed to prove Theorem \ref{thm - sufficient hardy} and Theorem \ref{thm - sufficient with rapidly decreasing}. Section \ref{sec - characterisation thms} contains the proofs of our main theorems. In Section \ref{sec - appendix} we give the proofs of the results in Section \ref{sec - Prelim results}, namely we prove two versions of the Calder\'on reproducing formulas on $\s'$ and use these to deduce growth rate for distributions in the Besov-Lipschitz spaces.\\

We conclude the introduction by a remark. All the main results presented in this paper have continuous versions, in which the sum is replaced by the integral with respect to the dilation variable $t>0$, and the kernel function satisfies the ``standard'' Tauberian condition (see \cite{Bui1996}). We leave the precise formulation as well as modification of the proofs to the interested reader, but note that details in the smooth kernel case can be found in \cite{Bui1996,Bui1997,Bui2000}. Moreover, versions of our results should also hold in the weighted case, where the parameter $\lambda$ in the Peetre maximal function will depend on the weight function $w$. Again we refer to the above cited works for a treatment in the smooth kernel case.\\

An earlier version of this paper was presented at the 2008 Australia-New Zealand Mathematics Convention \cite{Bui2008}.

\section{Preliminary Results}\label{sec - Prelim results}

We begin by recalling two versions of the Calder\'{o}n reproducing formula. These two theorems are classic results that were first used in the study of the  homogeneous Besov-Lipschitz spaces by Peetre \cite{Peetre1976}. (A continous version of Theorem \ref{thm - classical calderon on s'} was attributed to A.P. Calder\'on by the authors of \cite{Janson1981}.) We collect the proof of the two theorems below in the appendix for easy reference (see Subsection \ref{subsec - proof of classical calderon form} below). It is worth noting that our proofs are carried out in the spatial space and are different from \cite{Peetre1976}  (where it was done in the frequency domain). Moreover, our argument gives an explicit definition of the sequence of polynomials $(p_N)$ appearing in the theorems below; see equation (\ref{eqn - thm classical cald on s' - defn of poly}).

\begin{theorem}[Calder\'on Formula on $\s'$]\label{thm - classical calderon on s'}
 Let $f \in \s'$. Then there exists a sequence of polynomials $(p_N)$ such that
        $$ f= \lim_{N \rightarrow - \infty} \Big( p_N + \sum_{j = N + 1}^\infty \q_j * \q_j * f \Big)$$
 with convergence in $\s'$.
\end{theorem}

We can deduce a more refined version if we make the additional assumption that $f \in \dot{B}^\alpha_{\infty, \infty}$.

\begin{theorem}[Calder\'on Formula on $\dot{B}^\alpha_{\infty, \infty}$]\label{thm - classical calderon on B}
Let $\alpha \in \RR$ and $ f\in \dot{B}^\alpha_{\infty, \infty}$. Then there exist polynomials $p$, $p_N$ such that\footnote{If $\alpha<0$ this statement is vacuous, and we simply have $p_N=0$.} $\deg(p_N)\les[\alpha]$ and
        $$ f-p = \lim_{N\rightarrow - \infty}\bigg( p_N +  \sum_{j=N+1}^\infty \q_j * \q_j*f \bigg)$$
with convergence in $\s'$. Moreover, given any $\rho \in \s$ we have the inequality
        $$\sup_{N<0\les M} \bigg| \rho * \bigg( p_N + \sum_{j=N+1}^M \q_j * \q_j * f\bigg)(x) \bigg|\lesa 1 + \begin{cases}
          |x|^{\max\{0, \alpha\}}  \qquad &\alpha \not \in \NN \\
          |x|^\alpha \log(|x|) &\alpha \in \NN.
        \end{cases}$$
\end{theorem}

In the characterisation results presented in the current paper, we restrict our attention to distributions of finite growth. To see that this restriction is reasonable, we need to show that elements of $\dot{B}^\alpha_{p, q}$ and $\dot{F}^\alpha_{p, q}$ have growth of some finite order. This growth is a straight forward application of the bound in Theorem \ref{thm - classical calderon on B}.

\begin{corollary}[Growth of distributions in $\dot{B}^\alpha_{p, q}$ and $\dot{F}^\alpha_{p, q}$]\label{cor - growth bound on B and F}
Let $\alpha \in \RR$, $0<p, q \les \infty$. If $f \in \dot{B}^\alpha_{\infty, \infty}$ then there exists a polynomial $p$ such that for every $\rho \in \s$ we have
     $$ | \rho * (f-p)(x) | \lesa  1 + \begin{cases}
           |x|^\alpha \log{ |x| } \qquad &\alpha \in \NN \\
          |x|^{\max\{0, \alpha\}} &\alpha \not \in \NN.
        \end{cases}$$
Consequently, if $f \in \dot{B}^\alpha_{p, q}$ or $f \in \dot{F}^\alpha_{p, q}$, then there exists a polynomial $p$ such that $f-p$ is a distribution of growth $\ell$ for every $\ell > \alpha - \frac{n}{p}$ with $\ell\ge 0$.
\begin{proof}
  The growth bound on $\dot{B}^\alpha_{\infty, \infty}$ follows immediately from Theorem \ref{thm - classical calderon on B}. To conclude the proof, we simply recall the embedding $\dot{B}^{\alpha}_{p, q}, \dot{F}^\alpha_{p,q} \subset \dot{B}^{\alpha - \frac{n}{p}}_{\infty, \infty}$. (Note that $\dot{F}^\alpha_{\infty,q} \subset   \dot{B}^\alpha _{\infty, \infty}$ by an embedding theorem in \cite{Bui2000}.)
\end{proof}
\end{corollary}

\begin{remark}\label{growth of besov  function}
When $\alpha>0$, it is well-known that the characterisation of $\dot{B}^\alpha_{\infty,\infty}$ via differences  implies the stronger pointwise growth bound
$$ |(f-p)(x) | \lesa \begin{cases}

          |x|^\alpha \log{ |x| } &\alpha > 0 \text{ and } \alpha \in \NN \\
          |x|^\alpha   &\alpha>0 \text{ and } \alpha \not \in \NN,
        \end{cases}$$
from which the Corollary follows. On the other hand, in the case $\alpha = 0$, the growth bounds in Corollary \ref{cor - growth bound on B and F} and Theorem \ref{thm - classical calderon on B} appear to be new.
\end{remark}

As is common in the study of function spaces via the Calder\'{o}n formula, we require some control over convolutions of the form $\eta_k * \psi_j$ (see for instance the work of Heideman \cite{Heideman1974}). The precise dilation estimate we need is a refined version of \cite[Lemma 2.1]{Bui1996} (see also \cite[Lemma 1]{Rychkov1999a}).

\begin{lemma}\label{lem - main dilation estimate}
Let $m \in \RR, c>0$ and  $N \in \NN$. Suppose $\eta \in L^1$ with $\widehat{\eta} \in C^{N}(\RR^n)$ and $\supp \widehat{\eta} \subset \{ a\les |\xi| \les b\}$ for some $0<a<b$. Let $\psi \in \s'$ with $\widehat{\psi} \in C^{N }\big(\RR^n\setminus\{0\}\big)$.
\begin{enumerate}[label=(\roman*)]
\item Assume $\p^\kappa \widehat{\psi}(\xi)= \mathcal{O}( |\xi|^{ - m} )$ as $|\xi| \rightarrow \infty$ for every $|\kappa|\les N$. Then for any $s\les ct$ we have
            \begin{equation}\label{ver2 eqn1} |\eta_s*\psi_t(x)|\lesa \Big( \frac{s}{t}\Big)^{ m  - n } \, \frac{ t^{-n}}{( 1 + t^{-1} |x|)^N}. \end{equation}
\item Assume  $\p^\kappa \widehat{\psi}(\xi)= \mathcal{O}( |\xi|^{m - |\kappa|} )$ as $|\xi| \rightarrow 0$ for every $|\kappa|\les N$. Then for any $t\les cs$ we have
            \begin{equation}\label{ver2 eqn2}  |\eta_s*\psi_t(x)|\lesa \Big( \frac{t}{s}\Big)^{ m} \frac{ s^{-n}}{( 1 + s^{-1} |x|)^N}. \end{equation}
\end{enumerate}
\begin{proof}
Take $c=1$ for simplicity of notation.
The support assumption on $\widehat{\eta}$ implies that the convolution $\eta*\psi$ is well-defined (in fact is an $L^\infty$ function). Moreover,  for every $|\kappa| \les N$ and any $x \in \RR^n$
    \begin{eqnarray}\label{eqn - lem main dilation est - Linfty estimate}
     s^{n -|\kappa|} \big| x^\kappa\, \eta_s * \psi_t(x) \big| \les  \big\| x^\kappa \,\eta * \psi_{\frac{t}{s}}(x) \big\|_{L^\infty}
& \lesa & \big\| \p^\kappa \big[\widehat{\eta}(\xi) \widehat{\psi}_{\frac{t}{s}}(\xi)\big] \big\|_{L^1} \nonumber \\
& \lesa & \sum_{\gamma \les \kappa} \Big( \frac{t}{s} \Big)^{|\gamma|} \int_{a\les |\xi|\les b}  |\p^\gamma \widehat{\psi}( \tfrac{t}{s} \xi) \big| d \xi.
    \end{eqnarray}
In particular, if $s \les t$, then assuming $\p^\kappa \widehat{\psi} = \mc{O}(|\xi|^{ - m })$ as $|\xi| \rightarrow \infty$,  and using the bound (\ref{eqn - lem main dilation est - Linfty estimate}) we deduce that for every $|\kappa| \les N$
        \begin{eqnarray*}
             \,\big| (t^{ - 1}x)^\kappa\, \eta_s * \psi_t(x) \big| & \lesa  & s^{-n}  \Big( \frac{s}{t} \Big)^{|\kappa|} \sum_{\gamma \les \kappa} \Big( \frac{t}{s} \Big)^{|\gamma|} \int_{a\les |\xi| \les b}  \big| \tfrac{t}{s} \xi \big|^{ -m} d\xi  \\
& \approx & t^{-n} \Big( \frac{s}{t} \Big)^{ m - n  } \sum_{\gamma \les \kappa} \Big( \frac{s}{t} \Big)^{|\kappa| - |\gamma|}
\lesa   t^{-n} \Big( \frac{s}{t} \Big)^{ m - n  }.
        \end{eqnarray*}
Applying this estimate for $|\kappa|=0$ and $|\kappa|=N$ we obtain (i).

Similarly, if $t \les s$, then assuming $\p^\kappa \widehat{\psi} = \mc{O}(|\xi|^{m - |\kappa|})$ as $|\xi| \rightarrow 0$ and again applying the bound (\ref{eqn - lem main dilation est - Linfty estimate}) we have
    $$ s^{ - |\kappa|} \, \big| x^{\kappa} \, \eta_s * \psi_t(x) \big| \lesa s^{ - n} \sum_{\gamma \les \kappa} \Big( \frac{t}{s} \Big)^{|\gamma|} \int_{a\les |\xi|\les b} \big| \tfrac{t}{s} \xi\big|^{ m - |\gamma|}  d\xi   \lesa  s^{ - n} \Big( \frac{t}{s} \Big)^m $$
which gives (ii).
\end{proof}
\end{lemma}

\begin{remark}\label{rem - dilation est with Besov}
 It is possible to generalise the previous lemma in the following sense. Suppose $\| ( 1 + |x|)^N \eta \|_{L^1} < \infty$ and $\supp \widehat{\eta} \subset \{ a<|\xi|<b\}$ for some $0<a<b<\infty$. Then for any $j, k \in \ZZ$
 \begin{equation}
   \label{eqn - conv est gen}
   \big\| ( 1 + 2^{\min\{j, k\}} |x|)^N \eta_j *\psi_k \big\|_{L^p}
   \les C_\psi 2^{ k ( n - \frac{n}{p})} 2^{ - ( j - k) m}
 \end{equation}
where
        $$ C_\psi \lesa \begin{cases}
    \sup_{|\gamma| \les N} \| P_{\gtrsim 1} ( x^\gamma \psi) \|_{\dot{B}^m_{p, \infty}} \qquad & j \gtrsim k \\
     \sup_{|\gamma| \les N} \| P_{\lesa 1} ( x^\gamma \psi) \|_{\dot{B}^{m+|\gamma|}_{p, \infty}} \qquad & j \lesa k
   \end{cases} $$
and $P_{\gtrsim 1}$ denotes the restriction to frequencies $\gtrsim 1$, i.e. $\| P_{\gtrsim 1} f \|_{\dot{B}^\alpha_{p, \infty}} = \sup_{j \gtrsim 1} 2^{ j \alpha} \| \q_j * f \|_{L^p}$, $P_{\lesa 1}$ is defined similarly. Thus we may replace the assumptions in Lemma \ref{lem - main dilation estimate} by supposing that $\psi$ belongs to certain \emph{Poised spaces of Besov type} (c.f. the work of Peetre \cite{Peetre1975}). The inequality (\ref{eqn - conv est gen}) follow by an application of H\"{o}lder's inequality, together with the support assumption on $\eta$ to deduce that
    \begin{align*}
      \big\| ( 1 + 2^{\min\{j, k\}} &|x|)^N \eta_j *\psi_k \big\|_{L^p} \\
      &= 2^{ k( n - \frac{n}{p}) } \big\| ( 1 + 2^{\min\{j-k, 1\}} |x|)^N \eta_{j-k} * \psi \big\|_{L^p} \\
      &\lesa 2^{ k( n - \frac{n}{p}) } \sup_{|\kappa + \gamma| \les N}  \Big( 2^{\min\{ (j-k)|\gamma|, -(j-k)|\kappa|\}} \big\|  ( x^\kappa \eta)_{j-k} * ( x^\gamma \psi) \big\|_{L^p} \Big)\\
      &\lesa 2^{ k( n - \frac{n}{p}) }  2^{ - (j-k) m  } \| ( 1 + |x|)^N \eta \|_{L^1}  \sup_{|\gamma| \les N} \Big(\sup_{j' \approx j-k} 2^{ j' m  + \min\{ j', 1\} |\gamma| } \| \q_{j'} * ( x^\gamma \psi) \|_{L^p}\Big)
    \end{align*}
which gives (\ref{eqn - conv est gen}) by definition of $\dot{B}^m_{p, \infty}$.

To illustrate the connection between (\ref{eqn - conv est gen}) and Lemma \ref{lem - main dilation estimate}, note that in the case $p \g 2$, the assumptions on $\psi$ in Lemma \ref{lem - main dilation estimate} implies that the righthand side of (\ref{eqn - conv est gen}) is finite. More
precisely, if  $2 \les p \les \infty$ and $ \widehat{\psi} \in C^N( \{ |\xi| \gtrsim 1 \})$ with $|\p^\gamma \widehat{\psi}(\xi)| \lesa |\xi|^{ - m - n( 1 - \frac{1}{p})}$ for $|\xi| \gtrsim 1$ and  $|\gamma| \les N$, then an application of the  Hausdorff-Young inequality gives for every $|\gamma| \les N$
    $$  \big\| P_{\gtrsim 1}( x^\gamma \psi ) \big\|_{\dot{B}^m_{p, \infty}}  \lesa  \sup_{ j \gtrsim 1} 2^{jm} \big\| \p^\gamma \widehat{\psi} \big\|_{L^{p'}( |\xi| \approx 2^j)} \lesa \sup_{ j \gtrsim 1 }  2^{m j } \big\| |\xi|^{ -m - \frac{n}{p'}} \big\|_{L^{p'}( |\xi| \approx 2^j)} < \infty.$$
Similarly, if $\widehat{\psi} \in C^N( \RR^n \setminus\{0\})$ with $ |\p^\gamma \widehat{\psi}(\xi)| \lesa |\xi|^{ - m - n( 1 - \frac{1}{p})- |\gamma|}$ for $0<|\xi| \lesa 1$ and $|\gamma| \les N$, then
 \begin{eqnarray*}
\big\| P_{\lesa 1}( x^\gamma \psi ) \big\|_{\dot{B}^{ m + |\gamma|}_{p, \infty}} &  \lesa & \sup_{ j \lesa 1} 2^{(m+ |\gamma|) j} \big\| \p^\gamma \widehat{\psi} \big\|_{L^{p'}( |\xi| \approx 2^j)}< \infty.
\end{eqnarray*}
Thus in terms of conditions on $\psi$, (\ref{eqn - conv est gen}) implies Lemma \ref{lem - main dilation estimate}. On the other hand, the disadvantage of (\ref{eqn - conv est gen}) is that firstly in certain cases we need more decay on $\eta$, and secondly the conditions on $\psi$ are more difficult to verify. As our emphasis is on finding conditions on our kernel which are easy to establish, throughout this article we ignore this slight generalisation and instead make use of Lemma \ref{lem - main dilation estimate}.
\end{remark}

\begin{remark}\label{rem - dilation est via L2}
  A typical application of Lemma \ref{lem - main dilation estimate} would involve estimating $\| ( 1 + 2^j |x|)^\lambda \eta_j * \psi_k \|_{L^1}$ via an application of H\"{o}lder's inequality to obtain
    $$\| ( 1 + 2^j |x|)^\lambda \eta_j * \psi_k \|_{L^1} \lesa 2^{- j n} \| ( 1 + 2^j |x|)^{n + 1 + [\lambda]} \eta_j * \psi_k \|_{L^\infty} $$
  and then using the $L^\infty$ bound obtained in Lemma \ref{lem - main dilation estimate} which requires $\widehat{\eta} \in C^{n+1+[\lambda]}$. However this argument can be performed more efficiently by using Plancheral instead of the $\| u \|_{L^\infty} \les \| \widehat{u} \|_{L^1}$ bound used in the proof of Lemma \ref{lem - main dilation estimate}. In more detail, we can use
    $$\| ( 1 + 2^j |x|)^\lambda \eta_j * \psi_k \|_{L^1} \lesa \| ( 1 + 2^j|x|)^{[\frac{n}{2} + \lambda] +1} \eta_j * \psi_k \|_{L^2} \lesa \sup_{|\gamma| \les [\frac{n}{2} + \lambda] + 1} 2^{ |\gamma| j } \| \p_\xi^\gamma (\widehat{\eta}_j \widehat{\psi_k}) \|_{L^2}$$
  which, after following the proof of Lemma \ref{lem - main dilation estimate}, only requires $\widehat{\eta} \in C^{ [\frac{n}{2} + \lambda] + 1}$. To summerize, it is often possible to replace the assumption $\widehat{\eta} \in C^{n+1 +  [\lambda]}$ with $\widehat{\eta} \in C^{[\frac{n}{2} + \lambda] + 1}$. A similar comment applies to the differentiability condition on $\widehat{\psi}$.
\end{remark}

To apply our characterisations to Poisson like kernels, we need to estimate the spatial decay of $\mathcal{F}^{-1}(|\xi|^\beta e^{-|\xi|})$. The required decay is a consequence of the following corollary of Lemma \ref{lem - main dilation estimate}.
\begin{corollary}\label{cor - fourier decay at origin implies spatial decay}
  Let $r>\ell \g0 $ and $1 \les p < \infty$. Let $\psi \in L^p$ and assume $\supp \widehat{\psi} \subset \{ |\xi| \les 1\}$. Furthermore, suppose that $\widehat{\psi} \in C^{n + 1 + [\ell]}( \RR^n \setminus \{0 \})$
  with
                $$ \p^\kappa \widehat{\psi}(\xi) =  \mathcal{O}( |\xi|^{r -|\kappa|}) \qquad \qquad  \text{ as } |\xi| \rightarrow 0 $$
  for every $|\kappa| \les n + 1 + [\ell]$. Then $ (1+|x|)^{\ell} \psi \in L^1$ and moreover $\psi$ has $[\ell]$ vanishing moments.
\begin{proof}
We begin by observing that
       \begin{align}  \sum_{j \les 1} \big\| |x|^\ell (\q_j * \q_j * \psi) \big\|_{L^1}   &\lesa \sum_{j \les 1}  2^{ - j \ell} \big\| ( 1 + 2^j |x|)^\ell \q_j * \psi \big\|_{L^1} \notag \\
       &\lesa \sum_{j \lesa 1} 2^{ j(r-\ell)} < \infty \label{eqn - cor fourier decay at origin implies spatial decay - sum belongs to L1}
       \end{align}
where we used an application of (ii) in Lemma \ref{lem - main dilation estimate} (with $t =1$ and $s=2^{-j}$) to deduce that
        $$ \| ( 1+ 2^j|x|)^\ell \q_j * \psi \|_{L^1} \lesa 2^{ j r} 2^{jn} \| ( 1 + 2^j|x|)^{\ell - (n + 1 + [\ell])} \|_{L^1}\lesa 2^{ j r }.$$
On the other hand, the assumption $\psi \in L^p$ together with $\supp \widehat{\psi} \subset \{ |\xi| \les 1\}$ implies that we have the pointwise identity\footnote{As in the standard proof of the reproducing formula (see (\ref{eqn - thm cald on s' - phi ident}) in Section \ref{sec - appendix}), there exists $\phi \in \s$ such that
        $$ \psi(x) = \phi_M * \psi(x) +  \sum_{M\les j \les 1} \q_j * \q_j * \psi(x) $$
  where we used the support assumption on $\widehat{\psi}$. An application of H\"{o}lder's inequality gives
    $$ \| \phi_M* \psi\|_{L^\infty} \les \| \phi_M \|_{L^{p'}} \| \psi \|_{L^p} \lesa   2^{ M \frac{n}{p}}. $$
 and thus,  as $p<\infty$, (\ref{calderon - zero}) follows by letting $M \rightarrow - \infty$.}
\begin{equation}\label{calderon - zero}
  \psi(x) = \sum_{ j \les 1 } \q_j * \q_j * \psi(x)
\end{equation}
for a.e. $x \in \RR^n$ (in fact as $\psi$ is smooth, the identity holds for every $x \in \RR^n$). Consequently,  (\ref{eqn - cor fourier decay at origin implies spatial decay - sum belongs to L1}) implies that $|x|^\ell \psi \in L^1$. Therefore $\| \psi \|_{L^1} \lesa \| \psi \|_{L^p} + \| |x|^\ell \psi \|_{L^1}$ and so we deduce that $( 1+ |x|)^\ell \psi \in L^1$. Finally, to check that $\psi$ has $[\ell]$ vanishing moments, we simply note that $\widehat{\psi} \in C^{[\ell]}(\RR^n)$, and hence the decay condition gives $\p^\kappa \widehat{\psi}(0) = 0$ for every $|\kappa| \les [\ell]$. Together with the integrability $( 1 + |\cdot|)^\ell \psi \in L^1$, this implies that $\psi$ has $[\ell]$ vanishing moments as claimed.

\end{proof}
\end{corollary}

\begin{remark}
  The previous corollary can be improved somewhat by using Remark \ref{rem - dilation est via L2}. In particular, we can replace the assumption $\widehat{\psi} \in C^{n + 1 + [\ell]}(\RR^n\setminus \{0\})$ with the slightly weaker $\widehat{\psi} \in C^{[\frac{n}{2} + \ell]+1}(\RR^n \setminus \{ 0 \} )$.
\end{remark}
  Corollary \ref{cor - fourier decay at origin implies spatial decay} has an immediate application to the Poisson kernel.
\begin{corollary}\label{Poisson kernel - integrability}
Let $\beta \g 0$, and let $\widehat{\psi}(\xi) = |\xi|^\beta e^{ - |\xi|}$. Then $( 1 + |\cdot|)^\ell \psi \in L^1$ for every $\ell < \beta$.
\begin{proof}
Let $\chi\in \s$ such that $\supp\chi \subset \{|\xi| \les 1\}$, and $\chi =1$ in a neighbourhood of the origin. Write
$\widehat{\psi} = \widehat{\psi} \chi + (1-\chi) \widehat{\psi} = \widehat{\theta}+\widehat{\mu}$. Then $\mu \in \s$ and $\theta$ satisfies the assumptions of Corollary \ref{cor - fourier decay at origin implies spatial decay} with $r=\beta$.
\end{proof}
\end{corollary}

Finally we make use of the following elementary summation inequalities.

\begin{proposition}\label{prop - summation inequalities}
Fix $0< p , q \les \infty$ and let $f_k$ be a sequence of measurable functions.  If $ ( a_j )_{j \in \ZZ} \in \ell^{\min\{p, q, 1\}}(\ZZ)$
then we have
$$\Big( \sum_{k\in \ZZ} \big( \big\lVert \sum_{j \in \ZZ} a_{j-k} f_k \big\rVert_p \big)^q \Big)^{1/q} \lesa \Big( \sum_{j \in \ZZ} \big\lVert f_j \big\rVert^q_p \Big)^{1/q}.$$
Similarly if $( a_j)_{j \in \ZZ} \in \ell^{ \min\{  q, 1\}}(\ZZ)$ then
$$ \Big\lVert \Big( \sum_{j\in \ZZ} \big(\sum_{k \in \ZZ} |a_{j-k}  f_k| \big)^q \Big)^{1/q} \Big\rVert_p \lesa \Big\lVert \Big( \sum_{j \in \ZZ} |f_j|^q \Big)^{1/q} \Big\rVert_p. $$
\begin{proof}
This proposition is a folklore result.
The proof is based on Young's inequality and the inequality
    $$\Big( \sum_{j \in \ZZ} |b_j| \Big)^r\les \sum_{j \in \ZZ} |b_j|^r$$
which holds whenever $0<r\les 1$. We omit the details.
\end{proof}
\end{proposition}

\section{Pointwise Definition of the Convolution}\label{sec - pointwise defn of conv}

The introduction of distributions of finite growth, together with Definition \ref{Defn of convolution}, makes it possible to define the convolution $\psi * f$ as a distribution. However, the characterisation results in Theorems \ref{thm - main thm necessary cond} and \ref{thm - sufficient with rapidly decreasing} require a \emph{pointwise} definition. In this section we give two sets of sufficient conditions to ensure that $\psi*f \in L^1_{loc}$. The first is via what is essentially a duality argument exploiting the Calder\'{o}n reproducing formula given in Theorem \ref{thm - classical calderon on B}. This argument has the advantage that it requires very few assumptions on the kernel $\psi$. On the other hand it is only applicable in the case $f \in \dot{B}^\alpha_{\infty, \infty}$, and thus is not helpful in Theorem \ref{thm - sufficient with rapidly decreasing}. The second approach is much more general, and works for arbitrary distributions $ f \in \s'$, provided only that $f$ has finite growth. However it correspondingly requires much stronger conditions on the the kernel $\psi$. \\

\subsection{The case $f \in \dot{B}^\alpha_{\infty, \infty}$} The key result is the following.

\begin{theorem}\label{thm - calderon and the pointwise defn of conv}
  Let $\alpha, \ell \in \RR$ with $\ell \g0$ and $\ell>\alpha$. Let $f \in \dot{B}^\alpha_{\infty, \infty}$. Assume $\psi \in \dot{B}^{-\alpha}_{1, 1}$ with $(1+|x|)^\ell \psi \in L^1$. Let $p$ be the polynomial given by Theorem \ref{thm - classical calderon on B}. Then the distribution $\psi*(f-p)$ is a bounded continuous function and we have the identity
       \begin{equation}\label{eqn - thm cald and pointwise defn - formula in statement}\psi*(f-p)(x) = \sum_{j \in \ZZ} \psi*\q_j*\q_j*f(x)\end{equation}
  where the sum converges in $L^\infty$.
\begin{proof}
  Let $p$ be the polynomial given in Theorem \ref{thm - classical calderon on B}. The decay assumption on $\psi$ implies that the convolution $\psi*(f-p) \in \s'$. Define $g$ as
         $$ g(x) = \sum_{j \in \ZZ} \psi*\q_j*\q_j*f(x).$$
  The duality estimate $\sum_j |\psi*\q_j*\q_j*f(x)| \les \| f \|_{\dot{B}^{\alpha}_{\infty, \infty}} \| \psi \|_{\dot{B}^{-\alpha}_{1, 1}}$ implies that $g$ is a bounded continuous function. Thus the theorem would follow by showing that for every $\rho \in \s$
        \begin{equation}\label{eqn - thm calderon and pointwise defn of conv - key equality} \psi*(f-p)(\rho) = \int_{\RR^n} g(x) \rho(x) dx. \end{equation}
  To this end, by definition of the distribution $\psi*(f-p)$, together with the growth bound in Theorem \ref{thm - classical calderon on B},  the Dominated Convergence Theorem, and the decay condition on $\psi$, we have for any $\rho \in \s$
    \begin{align*}\psi*(f-p)(\rho) &= \int_{\RR^n} \widetilde{\rho}*(f-p)(x) \psi(x) dx \\
        &= \int_{\RR^n} \lim_{N\rightarrow -\infty} \widetilde{\rho}*\bigg( p_N + \sum_{j=N+1}^\infty \q_j*\q_j*f\bigg)(x) \psi(x) dx \\
        &=\lim_{N\rightarrow - \infty}  \bigg( \int_{\RR^n} \rho(x) \psi*p_N(x) dx + \sum_{j=N+1}^\infty \int_{\RR^n} \psi*\q_j*\q_j*f(x) \rho(x) dx\bigg).
    \end{align*}
  We claim that the assumptions on $\psi$ imply that $\psi$ has $[\alpha]$ vanishing moments\footnote{In fact, the following argument shows that if $(1+|x|)^{[\alpha]} \psi \in L^1$ and $\psi \in \dot{B}^{-\alpha}_{1, q}$ for some $q<\infty$, then $\psi$ has $[\alpha]$ vanishing moments.}, in other words $\int x^\gamma \psi = 0$ for every $|\gamma| \les [\alpha]$. Accepting this claim for the moment, we have $\psi*p_N=0$ for every $N$ and hence
            $$ \psi*(f-p)\big( \rho\big) = \sum_{j \in \ZZ} \int_{\RR^n} \psi*\q_j*\q_j*f(x) \rho(x) dx = \int_{\RR^n} g(x) \rho(x) dx$$
  where the last equality follows by the uniform convergence of the sum. Therefore (\ref{eqn - thm calderon and pointwise defn of conv - key equality}) follows as required.

  Thus it only remains to show that $\psi$ has $[\alpha]$ vanishing moments. If $\alpha<0$ there is nothing to prove so we may assume that $\alpha \g 0$. The decay assumption on $\psi$ implies that $\widehat{\psi} \in C^{[\alpha]}(\RR^n)$ and hence using the form of the Taylor series given in \cite{Folland1990} we can write
    $$\widehat{\psi}(\xi) = \sum_{|\gamma| \les [\alpha]} \frac{\xi^\gamma}{\gamma!} \p^\gamma \widehat{\psi}(0) + [\alpha] \sum_{|\gamma|=[\alpha]} \frac{\xi^\gamma}{\gamma!} \int_0^1 (1-t)^{[\alpha]-1} \big(\p^\gamma \widehat{\psi}(t\xi) - \p^\gamma \widehat{\psi}(0) \big) dt. $$
The continuity of $\p^\gamma \widehat{\psi}$ at the origin then implies that
    \begin{equation}\label{eqn - psi minus taylor bound}
    \widehat{\psi}(\xi) - \sum_{|\gamma|\les [\alpha]} \frac{\xi^\gamma}{\gamma!} \p^\gamma \widehat{\psi}(0) = o\big(|\xi|^{[\alpha]}\big).
    \end{equation}
On the other hand, given any $\xi \not = 0$ we have
    $$ |\widehat{\psi}(\xi)| \lesa |\xi|^\alpha \sum_{ j \les \log_2(|\xi|)} 2^{ - j \alpha} \sup_{2^{j-1}\les |\xi| \les 2^{j+1}} |\widehat{\psi}(\xi)| \lesa |\xi|^\alpha \sum_{ j \les \log_2(|\xi|)+1} 2^{ - j \alpha}  \| \q_j * \psi \|_{L^1}$$
and consequently as $\psi \in \dot{B}^{-\alpha}_{1, 1}$, we deduce that $\widehat{\psi}(\xi)=o\big(|\xi|^\alpha\big)$ as $|\xi| \rightarrow 0$. Together with the bound (\ref{eqn - psi minus taylor bound}) we obtain
            $$ \sum_{|\gamma|\les [\alpha]} \frac{\xi^\gamma}{\gamma!} \p^\gamma \widehat{\psi}(0) = o\big(|\xi|^{[\alpha]}\big)$$
which is only possible if $\p^\gamma \widehat{\psi}(0) = 0$ for every $|\gamma| \les [\alpha]$. Therefore $\psi$ has $[\alpha]$ vanishing moments as claimed.
\end{proof}
\end{theorem}

\begin{remark}
  Let $\alpha \in \RR$ and suppose $f \in \dot{B}^\alpha_{\infty, \infty}$ and $\psi \in \dot{B}^{-\alpha}_{1, 1}$. It is well known that $\dot{B}^{-\alpha}_{\infty, \infty}$ can be identified with the (topological) dual of $\dot{B}^\alpha_{1,  1}$ (see e.g., \cite{Bui1979, Peetre1976}). Thus $f$ is a continuous linear functional on $\dot{B}^{-\alpha}_{1, 1}$ and furthermore, we have the identity\footnote{More explicitly, let $\widehat{\mathcal{O}}_0$ be the collection of all $\phi \in \s$ such that $0 \not \in \supp \widehat{\phi}$. Then as $\widehat{\mathcal{O}}_0$ is dense in $\dot{B}^{-\alpha}_{1, 1}$ (see e.g. \cite{Herz1968,Bui1979}), there exists a sequence $\phi^{(k)} \in \widehat{\mathcal{O}}_0$ such that $\| \psi - \phi^{(k)} \|_{\dot{B}^{-\alpha}_{1, 1}} \rightarrow 0$. We then define
        $$ f(\psi) = \lim_{k \rightarrow \infty} f\big( \phi^{(k)} \big).$$
  It is easy to check that the limit is independent of the choice of sequence $\phi^{(k)}$, and moreover that the resulting linear functional is continuous (as a map from $\dot{B}^{-\alpha}_{1, 1}$ to $\CC$). In addition, an application of Theorem \ref{thm - classical calderon on B} shows that
        $$ f(\psi) = \lim_{k \rightarrow \infty} f(\phi^{(k)}) = \lim_{k \rightarrow \infty} \sum_{j \in \ZZ} \q_j * \q_j * f(\phi^{(k)}) = \sum_{j \in \ZZ} \q_j*\q_j*f(\psi)$$
  where the last line followed from the Dominate Convergence Theorem, the assumption $\psi \in \dot{B}^{-\alpha}_{1, 1}$, and we used the fact that every $\phi^{(k)} \in \mathcal{\widehat{\mc{O}}}_0$ has infinite vanishing moments (thus annihilates all polynomials).}
        $$ f(\psi) = \sum_{j \in \ZZ} \q_j * \q_j * f(\psi).$$
  Thus if we define a convolution $\psi*_d f(x)$ as
            $$ \psi*_d f(x) = f( \tau_x \widetilde{\psi})$$
  we immediately have the pointwise identity
        $$ \psi*_d f(x) = \sum_{j \in \ZZ} \psi* \q_j * \q_j * f(x)$$
  (the convolutions on the righthand side are the standard convolutions between $\s$ and $\s'$). Since the sum converges uniformly, we see that $\psi *_d f(x)$ is a continuous bounded function. Although this definition of the convolution almost immediately gives the result of Theorem \ref{thm - calderon and the pointwise defn of conv}, it has the drawback that is does not always agree with the standard definition of the convolution. In particular, if for instance $\psi \in L^1$ and $f \in L^\infty$ is a constant, then $\psi*_d f = 0$, however $\psi*f(x) =  c \int \psi(y) dy$.

  It is natural to ask when  the convolution defined directly via duality, $\psi*_d f$, agrees with the definition given in Definition \ref{Defn of convolution}. The solution is given by the previous theorem. More precisely, suppose we know in addition that $( 1 + |x|)^\ell \psi \in L^1$ for some $\ell > \alpha$, then for every $f \in \dot{B}^\alpha_{\infty, \infty}$ there exists a polynomial $p$ such that we have the pointwise identity
            $$ \psi*_d f(x) = \psi*(f-p)(x).$$
\end{remark}

\subsection{The general case $f \in \s'$} We now drop the assumption $f \in \dot{B}^\alpha_{\infty, \infty}$, and instead simply assume that $f$ is a distribution of growth $\ell$. Our goal is find conditions on $\psi$ such that the convolution $\psi*f$ defined in Definition \ref{Defn of convolution}, which belongs to $\s'$, is in fact an element of $L^1_{loc}$. One possible solution is to assume $\psi \in \s$, as then  $\psi *f \in C^\infty$. However this is far to strong for our purposes, as we would like our characterisation, and thus the pointwise definition, to apply in the case $\psi \not \in \s$. The way forward, as in the case of $f \in \dot{B}^{\alpha}_{\infty, \infty}$, is to study the convergence of the Calder\'{o}n reproducing formula. The first step in this direction is the following lemma.

\begin{lemma}\label{lem - smoothness of distributions}
Let $\ell \g 0$  and assume $f$ is a distribution of growth $\ell $. Then there exists $\beta \g 0$ depending on $f$, such that for every $\phi \in \s$ and $k \in \ZZ$ we have
        $$ |\phi_k * f(x) | \lesa  2^{|k| \beta}\, \big( 1+ |x|\big)^\ell.$$

    \begin{proof}

      Define the mapping $T : \s \rightarrow L^\infty_\ell$ by $T(\phi) = \phi * f$ where $L^\infty_\ell$ denotes the weighted $L^\infty$ space defined by
        $$L^\infty_\ell = \{g: \| g(x) (1+ |x|)^{-\ell} \|_{L^\infty_x} < \infty \}.$$
      Since $f$ is a distribution of growth $\ell$, the linear mapping $T$ is well-defined. We claim that $T$ is continuous. To see this note that an application of the Closed Graph Theorem (see, for instance,  Theorem 1 on page 79 of \cite{Yosida1995})
reduces the problem to proving that the graph of $T$
            $$\Big\{ \big(\phi, T(\phi) \big) \, \Big| \, \phi \in \s \Big\}$$
      is closed in $\s\times L^\infty_\ell$. Assume $\phi^{(j)}$ converges to $\phi$ in $\s$ and $T(\phi^{(j)})$ converges to some $g \in L^{\infty}_\ell$. Then for some $M>0$ we have
            \begin{align*} |T(\phi^{(j)} - \phi)(x)| = |(\phi^{(j)} - \phi)*f(x) | &\lesa \sum_{|\alpha|, \gamma \les M} \| \phi^{(j)} ( x - \cdot) - \phi(x - \cdot) \|_{\alpha, \gamma} \\
                    &\lesa (1 + |x|)^M \sum_{|\alpha|, \gamma \les M} \| \phi^{(j)}  - \phi \|_{\alpha, \gamma} \end{align*}
      and hence $T(\phi^{(j)})$ converges to $T(\phi)$ pointwise. Therefore we must have $T(\phi) = g \in L^{\infty}_\ell$ and so the graph of $T$ is closed. Consequently $T$ is continuous as claimed.

      The continuity of $T$ implies that we can bound $\| T(\phi) \|_{L^\infty_\ell}$ by a finite number of Schwartz norms of $\phi$ (see, for instance, Corollary 1 on page 43 of \cite{Yosida1995}). Thus there exists $M_1>0$ such that
        \begin{equation}\label{bound on T}
            \| T(\phi) \|_{L^\infty_\ell}= \|\phi*f\|_{L^\infty_\ell} \lesa \sum_{|\alpha|, \gamma \les M_1} \| \phi \|_{\alpha, \gamma}.
          \end{equation}
To complete the proof, we observe that a simple computation shows that $ \| \phi_k \|_{\alpha, \gamma} \lesa 2^{k (n + |\alpha| - |\gamma|)}$
and hence, using (\ref{bound on T}), we obtain
        $$ \frac{ | \phi_k *f (x) |}{(1  + |x|)^\ell} \lesa  2^{|k| \beta}$$
        for some (possibly large) $\beta \g 0$ as required.
    \end{proof}
\end{lemma}

We can now prove the following.

\begin{proposition}\label{prop - convolution is a function}
Let $\ell \g 0$. Suppose $( 1 + |\cdot|)^\ell \psi \in L^1$ such that $\widehat{\psi} \in C^{ n + 1 + [\ell]}(\RR^n \setminus \{0\})$ with $\p^\kappa \widehat{\psi}$ rapidly decreasing for every $|\kappa| \les n + 1 + [\ell]$. Let $f \in \s'$ be a distribution of growth $\ell$. Then for every $j \in \ZZ$ the convolution $\psi_j*f$ is a well-defined continuous function. Moreover, there exists $\beta = \beta(f) >0$ such that for every $x \in \RR^n$
            \begin{equation}\label{rapidly implies function eqn1} M_{\ell, \beta}(x, j) = \sup_{k\g j, y \in \RR^n} \frac{ |\psi_k * f (y)| }{ (1+ 2^j |x-y|)^\ell} 2^{\beta( j -k)} < \infty. \end{equation}
\begin{proof}
  Fix $j \in \ZZ$ and let $k \g j$. The assumptions on $f$ and $\psi$ imply that $\psi_k*f \in \s'$. Thus we can follow the standard proof of the Calderon reproducing formula to deduce the identity
       \begin{equation}\label{eqn - prop convolution is a function - half caleron formula} \psi_k*f = \phi_k*\psi_k*f + \sum_{a=k+1}^\infty \q_a*\q_a*\psi_k*f \end{equation}
  where the sum converges in the sense of $\s'$ (see (\ref{eqn - thm cald on s' - phi ident}) in the proof of Theorem \ref{thm - classical calderon on s'}). To show that $\psi_k*f$ is a continuous  function it suffices to prove that the sum converges in  $L^\infty_{loc}$. An application of Lemma \ref{lem - smoothness of distributions} shows that there exists $\beta \g 0$ such that for every $\rho \in \s$ and  $k \g j$
        \begin{equation}\label{eqn - prop  convolution is a function - smoothness of f} |\rho_k*f(x)| \lesa_j 2^{ \beta k} ( 1 + |x|)^\ell.\end{equation}
   Note that by (i) in Lemma \ref{lem - main dilation estimate}, the assumption that $\widehat{\psi}$ is rapidly decreasing together with the support of $\widehat{\q}$ implies that $|\q_a*\psi_k(x)| \lesa 2^{(k-a)(\beta + 1)} 2^{kn} ( 1 + 2^k|x|)^{-(n+1+[\ell])}$. Therefore using an application of (\ref{eqn - prop convolution is a function - smoothness of f}) we deduce the bound
         $$ |\q_a*\q_a*\psi_k*f(x)|  \lesa_j 2^{a \beta} \int_{\RR^n} |\q_a*\psi_k(y)| (1+|x-y|)^\ell dx \lesa 2^{ k( \beta + 1)} 2^{ -a}  ( 1 + |x|)^\ell$$
  and hence the sum in (\ref{eqn - prop convolution is a function - half caleron formula}) converges uniformly on compact sets. Consequently $\psi_k*f$ is a continuous function. Finally, to deduce the required bound, we note that after another application of (\ref{eqn - prop  convolution is a function - smoothness of f}) we have for every $k \g j$
    \begin{align*} |\psi_k*f(x)| &\les |\psi_k * \phi_k * f(x)| + \sum_{a>k} |\q_a* \psi_k * \q_a * f(x)| \\
    &\lesa_j 2^{k  \beta } ( 1 + |x|)^\ell  + 2^{ k (\beta + 1)} \sum_{a \g k} 2^{-a} ( 1 + |x|)^\ell \lesa 2^{ k \beta} ( 1 + |x|)^\ell
    \end{align*}
  which then gives (\ref{rapidly implies function eqn1}).
  \end{proof}
\end{proposition}

\begin{remark}\label{remark sharpness of rapidly decreasing assumption}
 Lemma \ref{lem - smoothness of distributions} assures us that for any distribution $f$ of growth $\ell$, there exists a $\beta>0$ such that $f$ satisfies the conditions of Proposition \ref{prop - convolution is a function}. Thus provided we have $\psi \in L^1$ satisfying $(1+|\cdot|)^\ell \psi(\cdot) \in L^1$ and $\widehat{\psi} \in C^{n+1 +\ell}\big(\RR^n\setminus\{0\}\big)$ with, for every $|\kappa|\les n+1+\ell$ and some $m>\beta$,
        $$\p^\kappa \widehat{\psi}(\xi)= \mathcal{O}( |\xi|^{-n - m } ) \qquad \text{ as } |\xi| \rightarrow \infty,$$
then the convolution $\psi*f$ is a continuous function. Unfortunately, we have no control over how large $\beta$ is. Thus if we only assume that $f$ is a distribution of (unspecified) finite growth, to ensure $\psi * f$ is a function, we need $\psi$ to satisfy the conditions of Proposition \ref{prop - convolution is a function} for \emph{every} $\beta$. In particular we need $\widehat{\psi}$ to be rapidly decreasing.

Moreover, some smoothness of $\psi$ is required too. For example, for $\psi*f$ to be a well-defined function for every $f\in \s'$ of growth $0$, we require $\psi$ to be smooth. To see this take any multi-index $\kappa$ and let $f = \p^\kappa \delta_0$, where $\delta_0$ is the Dirac Delta function at the origin. Then $f$ is a distribution of growth $0$ and by Definition \ref{Defn of convolution}, for any $\phi \in \s$ we must have
        $$ \psi*f(\phi) = (-1)^{|\kappa|} \int_{\RR^n} \psi(x) \p^\kappa \phi(x) dx .$$
In particular, if $\psi * f \in \s'$ is represented by a function $g \in L^1_{loc}$  then for every $\phi \in \s$
            $$ \int_{\RR^n} \psi(x) \p^\kappa \phi(x) dx = (-1)^{|\kappa|} \int_{\RR^n} g(x) \phi(x) dx.$$
In other words $\psi$ must have $\kappa$ distributional derivatives which are locally integrable. As we can choose $|\kappa|$ to be arbitrarily large, Sobolev embedding then shows that $\psi \in C^\infty$.

\end{remark}

\section{Maximal Inequalities}\label{sec - maximal ineq}

As in the seminal work of Fefferman and Stein \cite{FS1972}, and Peetre \cite{Peetre1975, Peetre1976}, the key step in the proof of our characterisation theorems is to obtain certain pointwise maximal inequalities relating $\psi_j*f$ and $\q_j*f$. More precisely,
assuming for the moment that the convolution $\psi_k * f \in L^1_{loc}$, our goal in this section is to prove an inequality of the form
\begin{equation}\label{eqn - rough maximal suff} (\q^*_j f(x))^r \lesa \sum_{k \gtrsim j} 2^{ \delta (j-k)} \int_{\RR^n} \frac{ |\psi_k*f(x-y)|^r}{( 1 + 2^k|y|)^{\lambda r}} 2^{kn} dy \end{equation}
for some $\delta>0$, $0<r< \infty$, and $\lambda$ is as in the definition of the Peetre maximal function (\ref{eqn - defn Peetre max func}). The argument used to prove (\ref{eqn - rough maximal suff})  follows a strategy of Str\"{o}mberg-Torchinsky \cite{Stromberg1989} together with a number of technical refinements. The first of which is the following extension of the Calder\'{o}n reproducing formula.

\begin{proposition}\label{prop - Calderon gen}
Let $ \ell \g 0$. Suppose $ \psi \in L^1$ satisfies the Tauberian condition
with $\widehat{\psi} \in C^{n+1+ [\ell]}\big( \RR^n \setminus\{0\}\big)$.
There exists $\widehat{\eta}, \widehat{\phi} \in C^{n+1+[\ell]}(\RR^n)$ such that for every $g \in L^1_{loc}$ with $g(x) = \mathcal{O}(|x|^\ell)$ we have for $k \in \ZZ$ and a.e. $x\in \RR^n$
                \begin{equation}\label{Calderon gen eqn2} g(x) = \phi_k * g(x) + \sum_{j = k+1}^\infty \eta_j * \psi_j* g(x).\end{equation}
Moreover $\supp \widehat{\phi}$ is compact, and $\supp \widehat{\eta}$ is contained in some annulus about the origin.
   \begin{proof}

We start by observing  that there exists an $\eta \in L^1$ satisfying the required conditions, such that for all $\xi \not = 0$
    \begin{equation}\label{thm gen calderon - sum eta psi = 1}
            \sum_{j \in \ZZ} \widehat{\eta}(2^{-j} \xi) \widehat{\psi}(2^{-j} \xi) = 1
    \end{equation}
The construction of $\eta$  is standard and follows from the following observation: There exist positive numbers $a,b,c$  with
$0<2a\les b$ such that for every $\xi \in \RR^n$ there exists $j \in \ZZ$ satisfying $a\les2^{-j} |\xi| \les b$ and
\begin{equation*}\label{thm gen calderon - psi nonvanishing in annulus} |\widehat{\psi}(2^{-j} \xi)|^2 \g c.\end{equation*}
We refer to \cite[Chapter V, Lemma 6]{Stromberg1989} for details of this contruction in the smooth case. The modification to the nonsmooth case has been carried out in the thesis \cite{Candy2008} (see also \cite{vanEssen2005}).

Define
            $$ \widehat{\phi}(\xi) = \begin{cases}
                                \sum_{ j \les 0} \widehat{\eta}(2^{-j} \xi) \widehat{\psi}(2^{-j} \xi) \qquad \qquad & \xi \neq 0 \\
                                1           &\xi=0.
            \end{cases}$$
It is easy to check that $\phi$ satisfies the required conditions and that  $\widehat{\phi} =1$ in a neighbourhood of the origin. Moreover we have for any $k, m \in \ZZ$ with $m>k$
           \begin{equation}\label{Calderon gen eqn1}  \phi_m  - \phi_k = \sum_{j=k+1}^m \eta_j*\psi_j. \end{equation}
Take any $g \in L^1_{loc}$ satisfying $g(x)= \mathcal{O}(|x|^\ell)$. Note that as $\widehat{\phi}, \widehat{\psi} \widehat{\eta} \in C^{n+1+[\ell]}(\RR^n)$ we have $|\phi|, |\psi*\eta| \lesa ( 1 + |x|)^{-(n+1+[\ell])}$
and hence the convolutions $\eta*\psi*g$ and $\phi *g$ are well defined. Moreover since $\phi_m$ forms an approximation to the identity we have $\lim_{m \rightarrow \infty} \phi_m *g(x) = g(x)$ for a.e. $x\in \RR^n$ (more precisely this holds at every Lebesgue point of $g$). Thus taking the convolution of $g$ with both sides of (\ref{Calderon gen eqn1}) and letting $m \rightarrow \infty$ proves the result.
\end{proof}
\end{proposition}

To prove the maximal function inequality (\ref{eqn - rough maximal suff}), we need to assume the boundedness of a particular auxiliary maximal function,  namely, the following variation of the Peetre maximal function
   \begin{equation}\label{maximal variant} M_{\lambda, m}(x, j) = \sup_{y \in \RR^n, k\g j } \frac{|\psi_k*f(y)| }{(1+2^j|x-y|)^\lambda} \, 2^{(j-k)m}.\end{equation}
Note that if $M_{\lambda, m}(x_0, j)$ is finite for some $x_0 \in \RR^n$, then we have $M_{\lambda, m}(x, j) < \infty$ for all $x\in \RR^n$.
With these definitions at hand we now prove the following theorem which is essentially a non-smooth and discrete version of Theorem 2a in \cite[page 61]{Stromberg1989} (see also \cite[Lemma 2]{Bui1997}).

\begin{theorem}\label{thm - Stromberg}
Let $0<r \les 1$, $0< \lambda < \infty$, $\ell \g 0$, and $m, \beta \in \RR$.  Assume $(1+|\cdot|)^\ell \psi(\cdot) \in L^1$ satisfies the Tauberian condition. Moreover, suppose that $\widehat{\psi} \in C^{\max\{  n +1 +[\ell], [\lambda] +1\}}( \RR^n \setminus \{0\})$ and for every $|\kappa| \les \max\{[\ell], [\lambda]\} + 1$ we have
        $$ \p^\kappa \widehat{\psi}(\xi) = \mathcal{O}\big(|\xi|^{-\max\{ m, \beta\} }\big) \qquad \text{ as $|\xi| \rightarrow \infty$ }.$$
Let $f$ be a distribution of growth $\ell$ such that for every $j \in \ZZ$ the distribution $\psi_j*f$ is a locally integrable function with
        $$ M_{\ell, m}(x, j) <\infty.$$
 Then we have the pointwise inequality
       \begin{equation}\label{Stromberg thm hardy version}\big( \psi_j^*f(x) \big)^r \lesa \sum_{k=j}^{\infty} 2^{(j-k)(\beta - \lambda)r} \int_{\RR^n} \frac{|\psi_k*f(x-y)|^r}{( 1 + 2^k|y|)^{\lambda r} } 2^{ kn} dy \end{equation}
with constant independent of $f$, $j$, $m$, $\ell$ and $x$.

\begin{proof}
Fix $u\g j$.   The assumption $M_{\ell, m}(x, j) < \infty$ implies that $\psi_u*f = \mathcal{O}(|x|^\ell)$. Therefore the Tauberian condition and Proposition \ref{prop - Calderon gen} give
    \begin{equation}\label{Stromberg eqn1}\psi_u*f(x) = \phi_u*\psi_u*f(x) + \sum_{k=u+1}^{\infty} \eta_k*\psi_k*\psi_u*f(x) \end{equation}
with $\widehat{\phi}, \widehat{\eta} \in C^{ \max\{n + 1 +[\ell], [\lambda] + 1\}}(\RR^n)$ and support of $\widehat{\eta}$ is contained in some annulus about the origin. An application of Lemma \ref{lem - main dilation estimate} gives
    $$ ( 1 + 2^u |x|)^\lambda |\eta_k * \psi_u(x)| \lesa 2^{-( k-u) \beta} 2^{k n}$$
and thus we have the bound
\begin{align*}
 |\eta_k*\psi_k * \psi_u * f(x)| &\lesa  \big\| ( 1 + 2^u |\cdot|)^\lambda \eta_k * \psi_u \big\|_{L^\infty} \big\| ( 1 + 2^u|x-\cdot|)^{-\lambda} \psi_k*f \big\|_{L^1}\\
&\lesa 2^{kn} 2^{- \beta ( k-u)} \big\| ( 1 + 2^u |x-\cdot|)^{-\lambda} \psi_k*f \big\|_{L^1}.
\end{align*}
 On the other hand, the decay on $\phi$ shows
        $$ |\phi_u*\psi_u*f(z)| \les \big\|  ( 1 + 2^u|\cdot|)^\lambda \phi_u \big\|_{L^\infty} \big\| ( 1 + 2^u |z-\cdot|)^{-\lambda} \psi_u * f \big\|_{L^1}\lesa 2^{un} \big\| ( 1 + 2^u |z-\cdot|)^{-\lambda} \psi_u * f \big\|_{L^1} $$
and hence via (\ref{Stromberg eqn1}) we obtain, for every $z \in \RR^n$ and any $u \g j$,
        $$|\psi_u*f(z)| \lesa 2^{(u-j)\beta} \sum_{k=u}^{\infty} 2^{(j-k)\beta}  \int_{\RR^n} \frac{|\psi_k*f(y)| }{ (1+ 2^u|z-y|)^{\lambda}} 2^{k n} dy$$
where the constant depends only on $\psi$, $\beta$, and $\lambda$ (in particular, it is independent of $f$, $j$, $\ell$, and $m$). Now, since $k \g u \g j$, we have
        \begin{align*}
            &\frac{ |\psi_k*f(y)|}{(1+2^u|z-y|)^{\lambda}} \,2^{(j-k)\beta}\\
                                    &\,\,\,\,\,\,\,\,=\Bigg( \frac{ |\psi_k*f(y)|}{(1+2^j|x-y|)^{\lambda}} \,2^{(j-k)\beta}\Bigg)^r\Bigg(\frac{|\psi_k*f(y)|}{(1+2^j|x-y|)^{\lambda}}\, 2^{(j-k)\beta} \Bigg)^{1-r}\frac{ (1+2^j|x-y|)^{\lambda} }{(1+2^u|z-y|)^{\lambda}} \\
                                    &\,\,\,\,\,\,\,\,\les \Bigg( \frac{ |\psi_k*f(y)|}{(1+2^j|x-y|)^{\lambda}} \,2^{(j-k)\beta}\Bigg)^r\, M_{\lambda, \beta}(x, j)^{1-r}\,
                                                            (1+2^j|x-z|)^{\lambda} \label{Max finite eqn3}\end{align*}
and hence using the elementary inequality $( 1 + 2^j |y|)^{-1} \lesa 2^{ k-j} (1 + 2^k |y|)^{-1} $ we deduce that
        \begin{align*} \frac{ |\psi_u*f(z)|}{(1+2^j|x-z|)^{\lambda}}\,\, 2^{(j-u)\beta}&\lesa M_{\lambda, \beta} (x, j)^{1-r} \sum_{k=u}^{\infty} 2^{(j-k)\beta r} \int_{\RR^n}\frac{ |\psi_k*f(y)|^r}{(1+2^j|x-y|)^{\lambda r}} 2^{kn} dy \\
                               &\lesa M_{\lambda, \beta} (x, j)^{1-r} \sum_{k=j}^{\infty} 2^{(j-k)(\beta -\lambda)r} \int_{\RR^n}\frac{ |\psi_k*f(y)|^r}{(1+2^k|x-y|)^{\lambda r}} 2^{kn} dy.\end{align*}
Thus taking the  supremum over $z \in \RR^n$ and $u \g j$ yields,
       \begin{equation}\label{Max finite eqn1} M_{\lambda, \beta }(x, j) \lesa M_{\lambda, \beta} (x, j)^{1-r} \sum_{k=j}^{\infty} 2^{(j-k)(\beta - \lambda) r} \int_{\RR^n}
                               \frac{ |\psi_k*f(y)|^r}{(1+2^k|x-y|)^{\lambda r}} 2^{kn} dy. \end{equation}
If we had $M_{\lambda, \beta}(x,j) < \infty$, then noting that $\psi_j^*f(x) \les M_{\lambda, \beta}(x, j)$, we obtain
    \begin{equation}\label{Stromberg} \big(\psi_j^*f(x)\big)^r \lesa  \sum_{k=j}^{\infty} 2^{(j-k)(\beta - \lambda) r}\int_{\RR^n}
                               \frac{ |\psi_k*f(y)|^r}{(1+2^k|x-y|)^{\lambda r}} 2^{kn} dy. \end{equation}
Note that the constant in (\ref{Stromberg}) is independent of $f$, $j$, $m$, $\ell$, and $x$.  Therefore it suffices to prove $M_{\lambda, \beta} <\infty$.

To this end let $m' = \max\{m, \beta\}$ and $  \lambda' = \max\{ \ell, \lambda\}$. Note that by our assumption we have $M_{\lambda', m'} \les M_{\ell, m}  <\infty$. Moreover, we have $( 1 + |\cdot|)^{\lambda'} \eta, (1 + |\cdot|)^{\lambda'} \phi \in L^\infty$ and via Lemma \ref{lem - main dilation estimate}
        $$ 2^{(k-u)m'} 2^{kn} ( 1 + 2^u |x|)^{\lambda'} |\eta_k * \psi_u(x)| < \infty.$$
Thus repeating the argument used to obtain (\ref{Stromberg}) (with $(\lambda,\beta)$ replaced by $(\lambda',m')$) we have
        \begin{equation}\label{Max finite eqn2} |\psi_u * f(y) |^r \les M_{\lambda', m'}^r(y, u) \lesa_m\sum_{k=u}^{\infty} 2^{(u-k)(m'r-n)} \int_{\RR^n}
                               \frac{ |\psi_k*f(z)|^r}{(1+2^u|y-z|)^{\lambda' r}} 2^{u n} dz.\end{equation}
Since the right hand side of (\ref{Max finite eqn2}) only gets larger if we decrease $m'$ and $\lambda'$, we deduce that (\ref{Max finite eqn2}) in fact holds for $\lambda'=\lambda$ and $m'= \beta$ (but with a constant that depends on $m$, hence this argument cannot be used to prove (\ref{Stromberg}) directly). Moreover, as
    $$ \frac{ 2^{(j-u)\beta r}}{ ( 1+ 2^j |x-y|)^{\lambda r} } \times \frac{2^{u n} }{  (1+ 2^u|y-z|)^{\lambda r}} \times 2^{(u-k)(\beta r - n)}
                    \les  2^{(j-k)(\beta r - n) }\,\frac{ 2^{jn} }{ (1+2^j|x-z|)^{\lambda r} } $$
we have for any $u\g j$
 \begin{align*}
        \frac{|\psi_u * f(y) |^r}{(1+ 2^j |x-y|)^{\lambda r}} \,2^{(j-u) \beta r} &\lesa \sum_{k=u}^{\infty} 2^{(j-k)(\beta r-n)} \int_{\RR^n}
                               \frac{ |\psi_k*f(z)|^r}{(1+2^j|x-z|)^{\lambda r}} 2^{jn} dz \\
                               &\lesa \sum_{k=j}^{\infty} 2^{(j-k)(\beta - \lambda) r} \int_{\RR^n}
                               \frac{ |\psi_k*f(z)|^r}{(1+2^k|x-z|)^{\lambda r}} 2^{kn} dz. \end{align*}
Therefore, provided the right hand side of (\ref{Stromberg}) is finite, we obtain $M_{\lambda, \beta}<\infty$ and so (\ref{Stromberg thm hardy version}) follows.
\end{proof}
\end{theorem}

The required maximal inequality (\ref{eqn - rough maximal suff}) is now a corollary of the previous Str\"{o}mberg-Torchinsky type estimate, Theorem \ref{thm - Stromberg}, together with another application of the Calder\'{o}n reproducing formula in Proposition \ref{prop - Calderon gen}.

\begin{corollary}\label{cor - max ineq suff}
Let $0<r, \lambda < \infty$, $\ell \g 0$, and $m, \beta \in \RR$.  Assume $(1+|\cdot|)^\ell \psi(\cdot) \in L^1$ satisfies the Tauberian condition. Moreover, suppose that $\widehat{\psi} \in C^{n+1 + \max\{ [\ell], [\lambda]\}}( \RR^n \setminus \{0\})$ and for every $|\kappa| \les \max\{[\ell], [\lambda]\} + 1$ we have
        \begin{equation}\label{eqn - cor max ineq suff - smoothness cond} \p^\kappa \widehat{\psi}(\xi) = \mathcal{O}\big(|\xi|^{-\max\{ m, \beta\} }\big) \qquad \text{ as $|\xi| \rightarrow \infty$ }.\end{equation}
Let $f$ be a distribution of growth $\ell$ such that for every $j \in \ZZ$ the distribution $\psi_j*f$ is a locally integrable function with
        $$ M_{\ell, m}(x, j) <\infty.$$
 Then we have the pointwise inequality
      $$\big( \q_j^*f(x) \big)^r \lesa \sum_{k \gtrsim j} 2^{(j-k)(\beta-\lambda) r} \int_{\RR^n} \frac{|\psi_k*f(x-y)|^r}{( 1 + 2^k|y|)^{\lambda r} } 2^{ kn} dy$$
with constant independent of $f$, $j$, $m$, $\ell$ and $x$.
\begin{proof}
Assume $f$ is  a distribution of growth $\ell$. Then $\q_j*f = \mathcal{O}(|x|^\ell)$ and so we can apply Proposition \ref{prop - Calderon gen} and obtain
            $$ \q_j*f(x) = \phi_u*\q_j*f(x) +\sum_{k = u+1}^\infty \eta_k*\q_j*\psi_k*f(x).$$
  where $\widehat{\eta}, \widehat{\phi} \in C^{n+1+\max\{[\ell], [\frac{n}{p}]\} }(\RR^n)$, $\supp \widehat{\phi} \subset \{ |\xi|<b\}$, and $\supp \widehat{\eta} \subset\{ a<|\xi|<b\}$ for some $a, b>0$. Since $\supp \widehat{\q} \subset \{ 2^{-1} \les |\xi| \les 2\}$, by choosing $u=j-s$ with $s$ sufficiently large we have $\phi_u*\q_j =0$. Similarly, perhaps choosing $s$ slightly larger $\eta_k*\q_j=0$ for $k>j+s$. Therefore we have
        \begin{equation}\label{eqn - cor max ineq suff - calderon} |\q_j*f(x) | \les \sum_{ k = j-s}^{j+s} |\eta_k*\q_j*\psi_k*f(x)|. \end{equation}
If $r \g 1$, we simply use an application of Holder's inequality together with (\ref{eqn - cor max ineq suff - calderon}) to deduce that
        \begin{align*}
          \frac{ |\q_j * f(x-y)|}{ ( 1 + 2^j |y|)^{\lambda}} &\lesa \sum_{ j \approx k } \int_{\RR^n} 2^{-j \frac{n}{r}} ( 1 + 2^j|z-y|)^{\lambda} |\eta_k*\q_j(z-y)| \times \frac{ |\psi_k * f( x -z)|}{( 1 + 2^j |z|)^\lambda } 2^{ j \frac{n}{r}} dz \\
          &\lesa \sum_{ j \approx k } \big\| ( 1 + 2^j |\cdot|)^{-\lambda} \psi_k * f( x- \cdot) 2^{ j \frac{n}{r}} \big\|_{L^r}
        \end{align*}
where we used the decay of $\q$.
The require inequality now follows by taking the sup over $y \in \RR^n$ and then taking $r^{th}$ powers of both sides. On the other hand, if $0<r<1$, a similar application of (\ref{eqn - cor max ineq suff - calderon}) gives
        $$  \q^*_j f(x) \lesa \sum_{j \approx k} \| ( 1 + 2^j |\cdot|)^\lambda \eta_k * \q_j \|_{L^1}\,\,  \psi^*_k f(x) \lesa  \sum_{j \approx k}  \psi^*_k f(x). $$
If we again take $r^{th}$ powers of both sides, then result follows directly from an application of Theorem \ref{thm - Stromberg}.
\end{proof}
\end{corollary}
\begin{remark}
  In the case $r\g1$, the proof of Corollary \ref{cor - max ineq suff} shows that the decay condition on $\widehat{\psi}$, (\ref{eqn - cor max ineq suff - smoothness cond}), is not needed. In fact we only need the the smoothness assumption $\widehat{\psi} \in C^{n+1+\max\{[\lambda], [\ell]\}}(\RR^n \setminus 0\}$ to ensure that the $\eta$ given by Proposition \ref{prop - Calderon gen} has sufficient decay.
\end{remark}

\begin{remark}
  A careful examination of the proof of Theorem \ref{thm - Stromberg} and Corollary \ref{cor - max ineq suff}, shows that we may replace the condition (\ref{eqn - cor max ineq suff - smoothness cond}) with the slightly weaker condition
            $$ \sup_{j \gtrsim 1} \Big( 2^{j \max\{ \beta, m\}} \big\| ( 1 + |x|)^{\max\{\lambda, \ell\}} \q_j * \psi \big\|_{L^\infty}\Big) <\infty$$
  (c.f. the ``poised spaces of Besov type'' introduced by Peetre in \cite{Peetre1975}). Alternatively, as in Remark \ref{rem - dilation est with Besov}, we may assume that
    $$ \sup_{ |\gamma| \les \max\{ [\lambda], [\ell]\}} \big\| P_{\gtrsim 1} ( x^\gamma \psi) \big\|_{\dot{B}^{\max\{ m, \beta\}}_{\infty, \infty}}<\infty.$$
\end{remark}

\section{Proof of Characterisation Theorems}\label{sec - characterisation thms}
In this section we give the proofs of our main results. We start with the sufficient direction, i.e. Theorems \ref{thm - sufficient hardy} and \ref{thm - sufficient with rapidly decreasing}. The first step is the following preliminary version of Theorem \ref{thm - sufficient with rapidly decreasing}.

\begin{theorem}\label{thm - sufficient Lp assuming maximal func finite}
  Let $0<p, q \les \infty$, $\alpha \in \RR$. Assume $\lambda>\Lambda \g 0$ and $\ell \g 0$. Let $f$ be a distribution of growth $\ell$ and $( 1 + |\cdot|)^\ell \psi \in L^1$ satisfying the following:
   \begin{enumerate}
     \item[\rm{(S1)}] the kernel $\psi$ satisfies the Tauberian condition and  we have $\widehat{\psi} \in {C^{n+1+\max\{[\ell], [\Lambda]\} }(\RR^n \setminus\{0\})}$;

     \item[\rm{(S2)}] there exists $m \g 0$ such that for every $j\in \ZZ$ the distribution $\psi_j * f$ is a locally integrable function with
                    $$ M_{\ell, m}(x, j) < \infty; $$

     \item[\rm{(S3)}] there exists $\beta>\Lambda - \alpha$ such that for every $|\gamma| \les \max\{ [\Lambda], [\ell]\} +1$
                    $$ \p^\gamma \widehat{\psi} = \mathcal{O}(|\xi|^{ -\max\{ \beta, m\}}) \qquad \text{ as $|\xi| \rightarrow \infty$.}$$

   \end{enumerate}
    If $\Lambda = \frac{n}{p}$ then
        $$ \Big( \sum_{j \in \ZZ} \big( 2^{j \alpha} \| \q_j^*f\|_{L^p} \big)^q \Big)^{\frac{1}{q}} \lesa \Big( \sum_{j \in \ZZ} \big( 2^{j \alpha} \| \psi_j *f\|_{L^p} \big)^q \Big)^{\frac{1}{q}}$$
    with constant independent of $m$ and $f$. Similarly if $\Lambda = \max\{ \frac{n}{p}, \frac{n}{q}\}$ and $p<\infty$ then
       $$\Big\| \Big( \sum_{j \in \ZZ} \big( 2^{j\alpha} \q_j^*f\big)^q \Big)^{\frac{1}{q}} \Big\|_{L^p} \lesa \Big\| \Big( \sum_{j \in \ZZ} \big( 2^{j\alpha} |\psi_j*f|\big)^q \Big)^{\frac{1}{q}} \Big\|_{L^p}$$
    and in the case $p=\infty$
       $$ \sup_Q \Big( \frac{1}{|Q|} \int_Q \sum_{j \g - \ell(Q)} \big( 2^{j \alpha} \q_j^* f(x))^q dx \Big)^\frac{1}{q} \lesa \sup_Q \Big( \frac{1}{|Q|} \int_Q \sum_{j \g - \ell(Q)}\big( 2^{j \alpha} |\psi_j* f(x)|)^q dx\Big)^\frac{1}{q}$$
    where again the implied constant is independent of $m$ and $f$. Note that when $q=\infty$, the previous inequality takes the form
$$ \sup_Q  \sup_{j \g - \ell(Q)} \frac{1}{|Q|} \int_Q   2^{j \alpha} \q_j^* f(x) dx  \lesa \sup_Q  \sup_{j \g - \ell(Q)} \frac{1}{|Q|} \int_Q  2^{j \alpha} |\psi_j* f(x)| dx,$$
where we require $\Lambda = n$.

\end{theorem}
\begin{proof}
The proof follows the arguments used in \cite{Bui1996, Bui1997, Bui2000}, with Theorem \ref{thm - Stromberg} replacing \cite[Lemma 2]{Bui1997}.
We only prove the Triebel-Lizorkin case as the Besov-Lipschitz case is similar. As the lefthand side of the inequalities only gets larger if we decrease $\lambda$, we may assume $\max\{\frac{n}{p}, \frac{n}{q}\} = \Lambda<\lambda < \min\{\alpha + \beta, [\Lambda]+1\}$. Choose $0< r< \min\{ p, q\}$ with $\max\{\frac{n}{p}, \frac{n}{q}\} < \frac{n}{r} < \lambda$. An application of Corollary \ref{cor - max ineq suff}, together with a decomposition of $\RR^n$ into annuli centred at $x$,  gives the pointwise inequality
    \begin{equation}\label{eqn - thm suff prelim - max ineq}
        \big( 2^{ j \alpha} \q^*_j f(x)\big)^r \lesa \sum_{ k \gtrsim 1} 2^{ -k( \beta + \alpha- \lambda )r} M((2^{(k+j)\alpha} |\psi_{k+j}*f|)^r)(x)
    \end{equation}
where $M(g) = \sup_{R>0} R^{-n} \int_{|y|<R} |g(x-y)| dy$ denotes the Hardy-Littlewood maximal function, and we used the elementary estimate $\int_{\RR^n} \frac{|g(x-y)|}{(1 + 2^j |y|)^N} 2^{jn} dy \lesa M(g)(x)$ which holds provided $N>n$. Therefore, as $\frac{q}{r}, \frac{p}{r} >1$, we deduce that
    \begin{align*}
      \Big\| \Big( \sum_{j \in \ZZ} \big( 2^{j\alpha} \q_j^*f\big)^q \Big)^{\frac{1}{q}} \Big\|_{L^p} &= \bigg\| \bigg( \sum_{j \in \ZZ} \Big( \big[2^{j\alpha} \q_j^*f\big]^r \Big)^\frac{q}{r} \bigg)^{\frac{r}{q}} \bigg\|_{L^{\frac{p}{r}}}^{\frac{1}{r}} \\
      &\lesa \bigg\| \bigg( \sum_{j \in \ZZ} \bigg[ \sum_{ k \gtrsim 1} 2^{ -k( \beta + \alpha -  \lambda)r}  M\big(|2^{(k+j)\alpha}\psi_{k+j}*f|^r\big)  \bigg]^\frac{q}{r} \bigg)^{\frac{r}{q}} \bigg\|_{L^{\frac{p}{r}}}^{\frac{1}{r}} \\
      &\lesa \bigg(\sum_{k \gtrsim 1} 2^{ - k ( \beta + \alpha - \lambda)r} \bigg\| \bigg( \sum_{j \in \ZZ} \Big[ M\big(|2^{j\alpha} \psi_j*f|^r\big)\Big]^\frac{q}{r} \bigg)^{\frac{r}{q}} \bigg\|_{L^{\frac{p}{r}}}\bigg)^{\frac{1}{r}}\\
      &\lesa \bigg\| \bigg( \sum_{j \in \ZZ} \big( 2^{j\alpha} |\psi_j * f|\big)^q\bigg)^\frac{1}{q} \bigg\|_{L^p}
    \end{align*}
where we used the assumption $\beta> \lambda - \alpha$ together with vector valued Hardy-Littlewood maximal inequality of Fefferman-Stein \cite{Fefferman1971}.

The argument in the case $p=\infty$ is slightly different and is of a more computational nature. Fix a dyadic cube $Q$ and let $x \in Q$. Assume first that $q<\infty$. An application of Corollary \ref{cor - max ineq suff} with $r=q$ gives
    \begin{align}
      \sum_{ j \g - \ell(Q)} \big( 2^{j \alpha} \q_j^* f(x)\big)^q &\lesa \sum_{j \g - \ell(Q)} \sum_{k \gtrsim 1} 2^{-k( \alpha + \beta - \lambda)q} \int_{\RR^n} \frac{ \big( 2^{(j+k) \alpha} |\psi_{j+k}*f( x- y)| \big)^q}{( 1 + 2^{ (j+k)}|y|)^{\lambda q}} 2^{ ( j + k) n} dy \notag\\
            &\lesa \sum_{j \gtrsim - \ell(Q)}\int_{\RR^n} \frac{ \big( 2^{j \alpha} |\psi_j * f(x-y)|\big)^q}{(1 + 2^j|y|)^{\lambda q}} 2^{jn} dy\notag\\ \notag
            &=  \sum_{j \gtrsim - \ell(Q)} \int_{|y|\les 2^{\ell(Q)} } \frac{ \big( 2^{j \alpha} |\psi_j * f(x-y)|\big)^q}{(1 + 2^j|y|)^{\lambda q}} 2^{jn} dy \\
            &\qquad \qquad + \sum_{j \gtrsim - \ell(Q)} \sum_{a \g 1} \int_{|y| \approx 2^{ a + \ell(Q)}} \frac{ \big( 2^{j \alpha} |\psi_j * f(x-y)|\big)^q}{(1 + 2^j|y|)^{\lambda q}} 2^{jn} dy. \label{eqn - thm suff prelim - p inf int decomp}
    \end{align}
To estimate the first term in (\ref{eqn - thm suff prelim - p inf int decomp}) we let $Q^*$ denote a dyadic cube with $\ell(Q^*) \approx \ell(Q)$ such that $y + Q \subset Q^*$ for every $|y|\les 2^{ \ell(Q)}$. A computation then shows that
    \begin{align*}
      \frac{1}{|Q|} \int_Q \sum_{j \gtrsim - \ell(Q)} \int_{|y|\les 2^{\ell(Q)}}
&\frac{ ( 2^{ j \alpha} |\psi_j * f(x - y) |)^q}{(1 + 2^j |y|)^{\lambda q}} 2^{jn} dy  \,dx\\
                                &\lesa \sum_{j \gtrsim - \ell(Q)} \int_{|y| \les 2^{ \ell(Q)}} \frac{2^{jn}}{( 1 + 2^j|y|)^{\lambda q}} \frac{1}{|Q^*|}\int_{ Q^*} \big( 2^{j \alpha} |\psi_j * f(x)|\big)^q dx \, dy \\
                                &\lesa \sup_{Q'} \bigg(\frac{1}{|Q'|} \int_{Q'} \sum_{j \g - \ell(Q')} \big( 2^{j \alpha} |\psi_j *f(x)|\big)^q dx\bigg).
    \end{align*}
Thus it only remains to control the second term in (\ref{eqn - thm suff prelim - p inf int decomp}). To this end, observe that for $j \gtrsim - \ell(Q)$, $a\g 1$, and $|y| \approx 2^{ a+ \ell(Q)}$ we have
    $$ \frac{2^{jn}}{( 1 + 2^j |y|)^{\lambda q}} \lesa 2^{ j( n-\lambda q)} 2^{ -( a + \ell(Q))\lambda q} \lesa 2^{ - a( \lambda q - n)} 2^{-( a + \ell(Q))n}$$
where we used the fact that $\lambda > \frac{n}{q}$. Therefore
    \begin{align*}
      \sum_{ j \gtrsim - \ell(Q)} \sum_{a \g 1} \int_{|y| \approx 2^{ a + \ell(Q)}} &\frac{ \big( 2^{j \alpha} |\psi_j * f(x-y)|\big)^q}{(1 + 2^j|y|)^{\lambda q}} 2^{jn} dy \\
                &\lesa \sum_{ a \g 1} 2^{ - a( \lambda q - n )} \frac{1}{ ( 2^{a + \ell(Q)})^n} \int_{|x-y| \lesa 2^{ a+ \ell(Q)}} \sum_{ j \gtrsim - \ell(Q) } \big( 2^{ j \alpha} |\psi_j * f(y)|\big)^q dy\\
                &\lesa \sup_{Q'} \bigg(\frac{1}{|Q'|} \int_{Q'} \sum_{j \g - \ell(Q')} \big( 2^{j \alpha} |\psi_j *f(x)|\big)^q dx\bigg).
    \end{align*}
These two estimates imply the required inequality when $q<\infty$.

The proof in the case $p=q=\infty$ is similar, in fact simpler,  so we shall be brief. Fix a dyadic cube $Q$ and let $x\in Q$ as above. Let $j\g -\ell(Q)$. Using Corollary \ref{cor - max ineq suff} with $r=1$ we get
$$
2^{j \alpha} \q_j^* f(x) \lesa
 \sum_{k \gtrsim j} 2^{-(k-j)( \alpha + \beta - \lambda)} \int_{\RR^n} \frac{  2^{k \alpha} |\psi_k( x- y)|}{( 1 + 2^k|y|)^{\lambda }} 2^{ k n} dy.
$$
It follows that, by decomposing the $y$-integral as before and noting that $\lambda>n$ in this case, one obtains
$$
\frac{1}{|Q|}\int_Q 2^{j \alpha} \q_j^* f(x) dx
\lesa \sup_{Q'} \sup_{k\g -\ell(Q')}\bigg(\frac{1}{|Q'|} \int_{Q'}   2^{k \alpha} |\psi_k *f(x)| dx\bigg).
$$
The proof of the theorem is thus complete.
\end{proof}

The proof of the $p = \infty$ case in Theorem \ref{thm - main thm necessary cond} requires the following corollary (c.f. the proof of Lemma 4 and 5 in the work of Rychkov \cite{Rychkov1999a}).

\begin{corollary}\label{cor - p=infty Petree maximal function est}
  Let $0<q<\infty$, $\lambda>\frac{n}{q}$,  and $\lambda>n$ when $q=\infty$. Let $k\in \ZZ$. Then for any dyadic cube $Q$ we have
\begin{align*}
 \bigg(\frac{1}{|Q|} \int_Q \sum_{ j \g -(\ell(Q) + k) } \big( 2^{ j \alpha} \q^*_j f(x)  \big)^q \, dx\bigg)^{1/q} & \lesa (1+|k|)^\frac{1}{q} \big\| f \big\|_{\dot{F}^\alpha_{\infty, q}}, \; q<\infty\\
 \sup_{ j \g -(\ell(Q) + k)}\frac{1}{|Q|}
\int_Q   2^{ j \alpha} \q^*_j f(x) \, dx & \lesa  \big\| f \big\|_{\dot{F}^\alpha_{\infty, \infty}}.
\end{align*}
\begin{proof}
Assume first that $q<\infty$.
  An application of Theorem \ref{thm - sufficient Lp assuming maximal func finite} with $\psi = \q$ (in which case the assumptions (S1), (S2), and (S3) clearly hold) gives
    $$\bigg(\sup_{Q'} \frac{1}{|Q'|} \int_{Q'} \sum_{ j \g - \ell(Q')} \big( 2^{ j \alpha} \q^*_j f(x)  \big)^q \, dx\bigg)^{1/q} \lesa \| f \|_{\dot{F}^\alpha_{\infty, q}}.$$
  Thus it is enough to show that for $j < - \ell(Q)$,
           \begin{equation}\label{eqn - cor p=infty Petree max func est - const on small cubes}\frac{1}{|Q|} \int_Q \big( \q^*_j f(x) \big)^q dx \lesa \frac{1}{|Q'|} \int_{Q'} \big( \q^*_j f(x) \big)^q  dx \end{equation}
  where $Q'$ is a dyadic cube with $Q \subset Q'$ and $\ell(Q') = -j$. To this end, note that if $x, x' \in Q'$, then $|x-x'| \les 2^{n-j}$ and hence for any $y \in \RR^n$ we have
        $$ ( 1 + 2^j |x' - y|)^\lambda \les ( 1 + 2^j |x' - x| + 2^j| x - y|)^\lambda \lesa  ( 1 + 2^j |x-y|)^\lambda.$$
Therefore the definition of $\q^*_jf$ implies that for any $x, x' \in Q'$ we have  $\q^*_jf(x) \lesa \q^*_jf(x')$. Consequently $\q^*_jf(x)$ is essentially constant on cubes of side lengths $< 2^{-j}$, in particular, we have (\ref{eqn - cor p=infty Petree max func est - const on small cubes}). Thus the result for $q<\infty$ follows. The modification when $q=\infty$ is done in a similar manner to the proof of Theorem \ref{thm - sufficient Lp assuming maximal func finite}.

\end{proof}
\end{corollary}

\subsection{Sufficient Conditions}\label{subsec - suff cond}

We now come to the proof of the sufficient direction of our characterisations which are now a straightforward consequence of Theorem \ref{thm - sufficient Lp assuming maximal func finite}.

\begin{proof}[Proof of Theorem \ref{thm - sufficient with rapidly decreasing}]
We reduce to checking the conditions (S1), (S2), and (S3). The condition (S1) is clear. An application of Proposition \ref{prop - convolution is a function} shows that there exists $m \g 0$ such that (S2) holds. Finally the rapid decay of $\p^\kappa \widehat{\psi}$ implies that (S3) holds.
\end{proof}

The proof of our characterisation with $L^p$ replaced with $H^p$, namely Theorem \ref{thm - sufficient hardy}, again follows from Theorem \ref{thm - sufficient Lp assuming maximal func finite}.

\begin{proof}[Proof of Theorem \ref{thm - sufficient hardy}]
Let $ f \in \s'$ be a distribution of growth $\ell$ and take $\phi \in \s$ with $\int \phi \not = 0$. The idea is to show that there exists an $s>0$ such that the kernel  $\phi_s*\psi$ satisfies the assumptions (S1), (S2), and (S3) of Theorem \ref{thm - sufficient Lp assuming maximal func finite}. To check (S1), note that by following the argument leading to (\ref{thm gen calderon - psi nonvanishing in annulus}), there exists $0<2a<b$ and $c>0$ such that for every $a \les t \les b$ and $\xi \in \sph^{n-1}$
                $$ |\widehat{\psi}( t \xi)| \g c. $$
Since $\phi \in \s$ and $\widehat{\phi} (0) = \int \phi \neq 0$, there exists $r>0$ such that $|\widehat{\phi}(\xi)| >0$ for $|\xi|<r$. Now as $\widehat{\phi_s}(\xi) = \widehat{\phi}( s \xi)$ we only need to choose $s< \frac{r}{a}$ to ensure that $\phi_s*\psi$ satisfies the Tauberian condition. Clearly the remaining conditions in (S1) are also satisfied.

To verify (S2), observe that since $\phi_s*f = \mathcal{O}(|x|^\ell)$, the convolution $(\phi_s * \psi)_j*f$ is well-defined. Furthermore, an application of Lemma \ref{lem - smoothness of distributions} shows that there exists $m$ such that
        $$ |\phi_{sk} * f(x) | \lesa 2^{ |k| m}  ( 1 + |x|)^\ell$$
which implies that for any $x \in \RR^n$, $j \in \ZZ$,
    $$ \sup_{ y \in \RR^n, k \g j} \frac{ |(\phi_s * \psi)_k * f(y)|}{ ( 1 + 2^j|x-y|)^\ell} 2^{(j-k)m} \lesa \sup_{ y \in \RR^n, k \g j}  \frac{ ( 1 + |y|)^\ell}{(1 + 2^j |x-y|)^\ell} 2^{ ( j - k)m} 2^{ |k| m} < \infty.$$
Thus (S2) holds. Finally, the rapid decay of $\p^\kappa \widehat{\phi}$ ensures that (S3) holds provided that $\p^\kappa \widehat{\psi}$ is slowly increasing as $|\xi| \rightarrow \infty$.

Therefore, we may apply Theorem \ref{thm - sufficient Lp assuming maximal func finite} together with the pointwise bound $|\q_j*f(x)| \les \q^*_jf(x)$, to deduce that
        \begin{align*}
          \| f \|_{\dot{B}^\alpha_{p, q}} &\lesa \bigg( \sum_{ j \in \ZZ} \Big( 2^{j \alpha} \big\| \phi_{sj}*\psi_j * f \big\|_{L^p}\Big)^q \bigg)^\frac{1}{q} \\
            &\lesa \bigg( \sum_{ j \in \ZZ} \Big( 2^{j \alpha} \big\| \sup_{t>0} |\phi_{t}*\psi_j * f| \big\|_{L^p}\Big)^q \bigg)^\frac{1}{q} \\
            &\lesa \bigg( \sum_{ j \in \ZZ} \Big( 2^{j \alpha} \big\| \psi_j * f \big\|_{H^p}\Big)^q \bigg)^\frac{1}{q}
        \end{align*}
where the last line follow from the $H^p$ charaterisation of Fefferman-Stein \cite{FS1972}. Similarly, the Triebel-Lizorkin case follows via
            \begin{align*} \| f\|_{\F} &\lesa \Big\| \Big( \sum_{j \in \ZZ} \big( 2^{j\alpha} |\phi_{sj}*\psi_j*f|\big)^q \Big)^{\frac{1}{q}} \Big\|_{L^p}\\
                                        &\lesa \Big\| \Big( \sum_{j \in \ZZ} \big( 2^{j\alpha} \sup_{t>0} |\phi_t*\psi_j*f|\big)^q \Big)^{\frac{1}{q}} \Big\|_{L^p}.
            \end{align*}
An identical computation gives the $p=\infty$ case. Thus the proof of Theorem \ref{thm - sufficient hardy} is complete.
\end{proof}

\begin{remark}
An alternative, more direct proof is possible of the Besov-Lipschitz case in Theorem \ref{thm - sufficient hardy}. The details are as follows.  Assume $f$ is  a distribution of growth $\ell$. Choose $\rho \in \s$ with $\widehat{\rho}(\xi) = 1$ for $\xi \in \supp \widehat{\q}$ and let $\frac{n}{p} < \lambda < [\frac{n}{p}] +1$. Then from (\ref{eqn - cor max ineq suff - calderon}) we deduce that
         \begin{align*}
           |\q_j*f(x) | &\les \sum_{ k = j-s}^{j+s} |\eta_k*\q_j*\rho_j*\psi_k*f(x)|\\
                        &\lesa \sum_{ k = j-s}^{j+s}  \sup_{y \in \RR^n} \frac{ |\rho_j*\psi_k*f(x-y)|}{(1+2^j|y|)^\lambda} \int_{\RR^n} |\eta_k*\q_j(y)| (1+ 2^j |y|)^\lambda dy \\
                        &\lesa \sum_{k = j-s}^{j+s} M^{**}_{\lambda}(\psi_k*f)(x)
         \end{align*}
  where
        $$ M^{**}_\lambda(g)(x) = \sup_{t>0, y \in \RR^n} \frac{ \rho_t*g(x-y)}{\Big( 1+ \frac{|y|}{t} \Big)^{\lambda}}$$
 is the  maximal function of Fefferman-Stein and we used the fact that $\eta(x) = \mathcal{O} (|x|^{-n-1-[\frac{n}{p}]})$. By the characterisation of $H^p$ by Fefferman-Stein \cite{FS1972} we have
            $$ \| M^{**}_\lambda g \|_{L^p} \lesa \| g\|_{H^p}$$
  provided $\lambda > \frac{n}{p}$. Therefore
            \begin{align*} \| f \|_{\B} =\Big( \sum_{j \in \ZZ} \big( 2^{j \alpha} \| \q_j*f\|_{L^p}\big)^q \Big)^{\frac{1}{q}}
                        &\lesa  \Big( \sum_{j \in \ZZ} \big( 2^{j \alpha} \big\|  M^{**}_{\lambda}(\psi_j*f) \big\|_{L^p}\big)^q \Big)^{\frac{1}{q}}\\
                        &\lesa \Big( \sum_{k \in \ZZ} \big( 2^{k \alpha} \| \psi_k*f\|_{H^p}\big)^q \Big)^{\frac{1}{q}}. \end{align*}
Hence (\ref{sufficient - hardy besov}) is proved.

 We note that a similar argument gives the corresponding Triebel-Lizorkin version as
      \begin{equation}\label{tangential maximal function triebel}\| f\|_{\F} \lesa \Big\| \Big( \sum_{j \in \ZZ} \big( 2^{j \alpha} M^{**}_\lambda(\psi_j*f) \big)^q \Big)^{\frac{1}{q}} \Big\|_{L^p} \end{equation}
 but this does not give  (\ref{sufficient - hardy triebel}) as there is no vector valued inequality relating $M^{**}_\lambda$ with $\sup_{t>0} |\phi_t *g|$. Thus we cannot directly deduce (\ref{sufficient - hardy triebel}) from (\ref{tangential maximal function triebel}) and instead need to argue via Theorem \ref{thm - sufficient Lp assuming maximal func finite}.
\end{remark}

Theorem \ref{thm - sufficient with rapidly decreasing} required $\p^\kappa \widehat{\psi}$ to be rapidly decreasing to ensure that the convolution $\psi*f$ was a locally integrable function. One way to avoid this fairly strong assumption on the kernel $\psi$,  was presented in Theorem \ref{thm - sufficient hardy} where we replaced the $L^p$ norm with the Hardy norm $H^p$ which is defined for elements of $\s'$. Consequently we only had to make sense of $\psi_j*f$ as an element of $\s'$ rather than $L^1_{loc}$. On the other hand, an alternative approach to finding a pointwise definition of the convolution is to instead make further assumptions on $f$. In particular, if we assume that $f$ is a slowly increasing function of order $\ell$, then the convolution $\psi*f$ is well defined as a function without the rapidly decreasing assumption. This leads to the following version of Theorem \ref{thm - sufficient Lp assuming maximal func finite}.

\begin{theorem} \label{thm - sufficient with f slowly increasing}
  Let $0<p, q \les \infty$, $\alpha \in \RR$,  and $\ell \g 0$. Let $\Lambda \g 0$ and $\beta> \Lambda - \alpha$. Assume  $(1 + |\cdot|)^{-\ell} f \in L^\infty$.
Suppose $\psi \in L^1$ satisfies the Tauberian condition with $(1+|\cdot|)^{\ell} \psi(\cdot) \in L^1$. Furthermore,  assume that $\widehat{\psi} \in C^{n+1 + \max\{ [\ell],  [\Lambda]\} }(\RR^n \setminus\{0\})$ with
    $$\p^\kappa \widehat{\psi}(\xi) = \mathcal{O} \big( |\xi|^{- \max\{\beta, 0\} }\big) \quad \text{as} \; |\xi| \to \infty$$
for $|\kappa| \les \max\{ [\Lambda] , [\ell]\} +1 $. If  $\Lambda = \frac{n}{p}$ then
        $$ \| f\|_{\B} \lesa \Big( \sum_{j \in \ZZ} \big( 2^{j \alpha} \| \psi_j *f\|_{L^p} \big)^q \Big)^{\frac{1}{q}}$$
Similarly if $ \Lambda = \max\{ \frac{n}{p} , \frac{n}{q}  \}$ and $p<\infty$ then
         $$ \| f\|_{\F} \lesa \Big\| \Big( \sum_{j \in \ZZ} \big( 2^{j\alpha} |\psi_j*f|\big)^q \Big)^{\frac{1}{q}} \Big\|_{L^p}$$
and in the case $p=\infty$
       $$ \| f\|_{\dot{F}^\alpha_{\infty,q}}
 \lesa \sup_Q \Big( \frac{1}{|Q|} \int_Q \sum_{j \g - \ell(Q)}\big( 2^{j \alpha} |\psi_j* f(x)|)^q dx\Big)^\frac{1}{q},$$
with the usual interpretation when $q=\infty$ (in which case $\Lambda = n$).
\begin{proof}
We begin by  observing that for $k\g0$
    $$|\psi_k*f(x)| \lesa (1+|x|)^{\ell}$$
which implies that $M_{\ell, 0}(x, 0)<\infty$ and consequently $M_{\ell, 0}(x, j)<\infty$ for every $j \in \ZZ$. Therefore result follows from Theorem \ref{thm - sufficient Lp assuming maximal func finite}.
\end{proof}
\end{theorem}

\begin{remark}
Assume $\alpha > n/p$. Then elements in $\B$ or in $\F$  are functions that satisfy the growth condition in Theorem \ref{thm - sufficient with f slowly increasing} with $\ell = \alpha-n/p$ when $\alpha-n/p \notin \NN$,  and $\ell > \alpha-n/p$ when $\alpha-n/p \in \NN$ (see Remark \ref{growth of besov function}). Hence this theorem readily gives the characterisation of these function spaces without the rapidly deacreasing assumption on the Fourier transform of the  kernel $\widehat{\psi}$.
\end{remark}

\subsection{Necessary Conditions}\label{subsection - Nec Cond}

We now come to the necessary direction of our characterisation, namely the proof of Theorem \ref{thm - main thm necessary cond}. As in the proof of Theorem \ref{thm - sufficient Lp assuming maximal func finite}, we follow the maximal function arguments used in the work of Bui-Paluszynski-Taibleson \cite{Bui1996, Bui1997,Bui2000}. In addition, in the case $p=\infty$, we rely also on an argument due to Rychkov \cite{Rychkov1999a}.

\begin{proof}[Proof of Theorem \ref{thm - main thm necessary cond}]
We first show that $\psi\in \dot{B}^{\frac{n}{p}-\alpha}_{1,1}$. To this end, the assumptions on $\psi$ together with Lemma \ref{lem - main dilation estimate} imply that
 \begin{equation}\label{proof of necessary - convolution estimate}
|\q_j * \psi(x)| \lesa \begin{cases}
          2^{ -jm} \frac{1}{(1 + |x|)^{ n + 1 + [\Lambda]} }\qquad & j \g 0 \\
          2^{ j r } \frac{ 2^{ j n}}{( 1 + 2^j |x|)^{n+1+[\Lambda]}} & j\les 0.
        \end{cases}
\end{equation}
It follows that
$$
	\|\psi\|_{\dot{B}^{\frac{n}{p}-\alpha}_{1,1}} \lesa \sum_{j\g 0} 2^{-j(\alpha-\frac{n}{p}+m)}
+ \sum_{j<0} 2^{j(r-\alpha+\frac{n}{p})} < \infty\quad (\text{as}\; \alpha-n/p+m>0,r>\alpha).
$$

Let $f\in \B$ or $f\in \F$. Then $f \in \dot{B}^{\alpha - \frac{n}{p}}_{\infty, \infty}$ by a well-known embedding theorem and hence by Theorem \ref{thm - classical calderon on B} and Theorem \ref{thm - calderon and the pointwise defn of conv}, after possibly subtracting a polynomial $\rho$, we see that $f$ is a distribution of growth $\ell$, the distribution $\psi_j*f$ is in fact a bounded continuous function, and moreover for every $x \in \RR^n$ we have the pointwise identity
            $$ \psi_k* f(x) = \sum_{j \in \ZZ^n} \q_j *\q_j*\psi_k*f(x).$$
Since $\psi^*_k f$ only gets smaller if we increase $\lambda$ we may assume $\Lambda <\lambda < \min\{ m + \alpha, [\Lambda] + 1\}$ where $\Lambda = \frac{n}{p}$ in the Besov-Lipschitz case, and $\Lambda = \max\{ \frac{n}{p}, \frac{n}{q}\}$ in the Triebel-Lizorkin case (this is possible as we assume that $m> \Lambda - \alpha$). If we now use an application of the above  Calder\'on formula we obtain
        \begin{align*}
          \frac{|\psi_k *f(x-z)|}{(1+ 2^k |z|)^\lambda}
            &\les \sum_{j \in \ZZ} \int_\RR \frac{ |\psi_k*\q_j(y)| }{(1+ 2^k |z|)^\lambda} |\q_j*f(x-z-y)| dy \\
            &\les \sum_{ j \in \ZZ} |\q_j^*f(x)| \int_{\RR} |\psi_k*\q_j(y)| \frac{(1+2^j |z+y|)^\lambda  }{(1+ 2^k |z|)^\lambda} dy.
        \end{align*}
A change of variables shows that
    $$ \int_{\RR} |\psi_k*\q_j(y)| \frac{(1+2^j |z+y|)^\lambda  }{(1+ 2^k |z|)^\lambda} dy = \int_{\RR} |\psi*\q_{ j - k}(y)| \frac{(1+2^{j-k} |2^k z+y|)^\lambda  }{(1+ 2^k |z|)^\lambda} dy $$
and hence we have the pointwise estimate
 \begin{equation}\label{proof of necessary cond - estimate on peetre maximal function}
   2^{k\alpha} \psi^*_k f(x) \lesa \sum_{j \in \ZZ} a_{j - k } 2^{j \alpha} \q_j^*f(x)
  \end{equation}
where
        $$ a_{j} = 2^{ - j\alpha } \sup_{x \in \RR^n} \int_{\RR^n} \big| \psi * \q_j( y) \big| \frac{ ( 1 + 2^j | x + y|)^\lambda}{( 1 + |x|)^\lambda} dy.$$
Thus, provided $(a_j) \in \ell^{ \min\{ p, q, 1\}}$ ( $(a_j) \in \ell^{\min\{ q, 1\}}$ in the Triebel-Lizorkin case), the first part of Theorem \ref{thm - main thm necessary cond}, (\ref{necessary  besov}) and (\ref{necessary triebel}), follows from Proposition \ref{prop - summation inequalities} together with the maximal function characterisation of Peetre \cite{Peetre1975}. In fact, using the obvious  inequality
            $$ \frac{ ( 1 + 2^j| x + y|)}{(1 + |x|)} \lesa \begin{cases} 2^{  j} ( 1 + |y|)  \qquad &j \g 0 \\
                                                                (1 + 2^j |y|) &j \les 0 \end{cases}$$
together with the estimates (\ref{proof of necessary - convolution estimate}), we get
\begin{equation}\label{eqn - proof of necessary - estimate on aj}
a_j \lesa \begin{cases} 2^{j(\lambda-m-\alpha)}\int_{\RR^n}(1+|y|)^{\lambda-n-1-[\Lambda]}\,dy, \qquad j\g 0\\
2^{j(r-\alpha)}\int_{\RR^n}2^{jn}(1+|2^j y|)^
{-n-1-[\Lambda]}\,dy, \qquad j\les 0. \end{cases}\end{equation}
By our assumptions, $\lambda>[\Lambda]\g 0, r>\alpha$, and $\lambda-m-\alpha < 0$ by our choice of $\lambda$, we deduce that $(a_j)\in \ell^\beta$ for all $\beta> 0$. Hence (\ref{necessary besov}) and (\ref{necessary triebel}) are proved.

It remains to consider the case $p=\infty$. We provide detail only in the case $q<\infty$ as the modification when $q=\infty$ is familiar by now. As in the case $p<\infty$,  an application of (\ref{proof of necessary cond - estimate on peetre maximal function}) together with (\ref{eqn - proof of necessary - estimate on aj}) shows that there exists $\delta>0$ such that
        $$ 2^{ k \alpha}  \psi^*_k f(x) \lesa \sum_{j \in \ZZ} 2^{ - |j| \delta} 2^{ (k-j) \alpha} \q^*_{k-j} f(x). $$
Therefore, Corollary \ref{cor - p=infty Petree maximal function est} gives for any dyadic cube $Q$
\begin{align*}
      \frac{1}{|Q|}\int_Q \sum_{k \g - \ell(Q) } \big( 2^{ k \alpha} \psi^*_k f(x) \big)^q dx &\lesa \frac{1}{|Q|} \int_Q \sum_{ k \g - \ell(Q)} \bigg( \sum_{ j \in \ZZ} 2^{ - |j| \delta} 2^{ ( k -j) \alpha} \q^*_{k-j} f(x) \bigg)^q dx \\
  &\lesa  \sum_{j \in \ZZ} 2^{ - |j| \delta \min\{ 1, q\}}  \bigg( \frac{1}{|Q|} \int_Q  \sum_{ k \g - ( \ell(Q) + j)} \big( 2^{ k \alpha} \q^*_kf(x) \big)^q dx \bigg) \\
  &\lesa \bigg( \sum_{j \in \ZZ} ( 1 + |j|) 2^{ -|j| \delta \min\{ 1, q\}} \bigg) \| f \|^q_{\dot{F}^\alpha_{\infty, q}} \\
  &\lesa \| f \|^q_{\dot{F}^\alpha_{\infty,q}},
\end{align*}
where, in the second inequality, we also use the $q$-triangle inequality when $q\les1$ and H\"{o}lder's inequality when $q> 1$.
Thus  (\ref{necessary triebel p=infty}) follows.

We now turn to the proof of (\ref{necessary hardy besov}), (\ref{necessary hardy triebel}), and (\ref{necessary hardy triebel p=infty}). Let $\phi \in \s$. Since $\phi_t*f$ satisfies the same properties as $f$ ($\phi_t*f \in \B$), we can repeat the proof of (\ref{proof of necessary cond - estimate on peetre maximal function}) to deduce that
        $$2^{k\alpha} |\phi_t * \psi_k*f(x)| \les 2^{k\alpha}\psi^*_k(\phi_t*f)(x) \lesa \sum_{j \in \ZZ} a_{k-j} 2^{j\alpha} \q_j^*(\phi_t*f)(x)$$
with constant independent of $t$. If we now follow the arguments leading to (\ref{necessary besov}), (\ref{necessary triebel}), and (\ref{necessary triebel p=infty}), it suffices to show that
            \begin{equation}\label{necessary eqn3} \q_j^*(\phi_t*f)(x) \lesa \q_j^*f(x).\end{equation}
To this end, let $\mu \in \s$ with $\widehat{\mu} = 1$ for $2^{-1} < |\xi| < 2$ and $ \sup \widehat{ \mu} \subset \{ 2^{-2} <|\xi|<4\}$. Then $\q_j = \q_j * \mu_j$ and so
        \begin{align*}
          \frac{|\q_j*\phi_t*f(x-y)|}{(1+ 2^j|y|)^\lambda} \lesa \q_j^*f(x) \int_{\RR^n} |\mu_j*\phi_t(z) | (1+ 2^j |z|)^\lambda dz.
        \end{align*}
If $2^{j} <\frac{1}{t}$ then (\ref{necessary eqn3}) follows easily by changing the order of integration. On the other hand if $2^{j} \g \frac{1}{t}$, then since $\phi \in \s$ we use part (i) of Lemma \ref{lem - main dilation estimate} to deduce that
        $$ |\phi_t*\mu_j(x)| \lesa \Big( \frac{2^{-j} }{t } \Big)^{n+1+[\lambda]} \frac{ t^{-n} }{ \Big(1+ \frac{|x|}{t}\Big)^{n+[\lambda]+1}} \lesa \frac{2^{jn}}{(1+ 2^j |x|)^{n+1+[\lambda]}}.$$
Therefore (\ref{necessary eqn3}) follows and so we obtain (\ref{necessary hardy besov}), (\ref{necessary hardy triebel}), and (\ref{necessary hardy triebel p=infty}).\\

\end{proof}

\section{Appendix}\label{sec - appendix}

\subsection{Proof of Theorems \ref{thm - classical calderon on s'} and \ref{thm - classical calderon on B}}\label{subsec - proof of classical calderon form}

\begin{proof}[Proof of Theorem \ref{thm - classical calderon on s'}]
Define $\phi \in \s$ by letting $\widehat{\phi}(\xi) = \sum_{j \les 0 } \widehat{\q}^2( 2^{- j} \xi) $ for $\xi \not = 0$ and $\widehat{\phi}(0) = 1$. Then  $\widehat{\phi}(\xi) = 1$ for $|\xi| \les 1$, $\widehat{\phi}(\xi) = 0$ for $|\xi|>2$ and moreover for any $ N<0<M$ we have the identity
        \begin{equation}\label{eqn - thm cald on s' - phi ident} \sum_{j = N+1}^M \q_j * \q_j(x) = \phi_M(x)  - \phi_N(x). \end{equation}
Let $f \in \s'$. Then there exists $a>0$ such that for every $\rho \in \s$
    $$ \big| f ( \rho) \big| \lesa \sum_{|\alpha|, |\beta| \les a } \| \rho\|_{\alpha, \beta}$$
where $\| \rho\|_{\alpha, \beta } = \sup_x \big| x^\alpha \p^\beta \rho(x) \big|$. In particular, for $N<0$ and any $\kappa$ we have
\begin{align*} \big| \p^\kappa\big( \phi_N*f\big)(  x ) \big| &= 2^{ N |\kappa|} \big| \big(\p^\kappa \phi\big)_N*f(  x ) \big| \\
&\lesa 2^{N |\kappa|} \sum_{|\alpha|, |\beta| \les a} \sup_{y \in \RR^n} \big| y^\alpha \p^\beta \big( \p^\kappa \phi\big)_N( x - y) \big|\\
 &\lesa  ( 1 + |x|)^a \,2^{ N( |\kappa| + n - a)}.\end{align*}
For $N<0$, we define the polynomial $p_N(x)$ as
        \begin{equation}\label{eqn - thm classical cald on s' - defn of poly} p_N(x) =  \sum_{ |\kappa|\les a - n} \frac{ \p^\kappa \big(\phi_N*f\big)(0)}{\kappa !}  x^\kappa \end{equation}
(if $a<n$ we can just take $p_N = 0$ for every $N$). By expanding $\phi_N*f$ as a Taylor series about $x=0$ and using the bound on $\p^\kappa (\phi_N*f)$ obtained above, we have
    \begin{align*}
      \big| \phi_N*f(x) - p_N(x) \big|&\lesa  |x|^{a + 1 - n} \sum_{|\kappa| = a - n + 1} \int_0^1  \big| \p^\kappa\big( \phi_N*f\big)( t x ) \big| dt \\
                                        &\lesa ( 1 + |x|)^{ 2a + 1 - n} 2^{ N}.
    \end{align*}
Consequently we see that $\phi_N*f - p_N \rightarrow 0$ in $\s'$ as $N \rightarrow -\infty$. On the other hand, since $\int \phi = 1$, we have $\phi_M * f \rightarrow f$ in $\s'$ as $M \rightarrow \infty$. Therefore, the identity (\ref{eqn - thm cald on s' - phi ident}) gives
    $$  \lim_{N \rightarrow -\infty} \Big( p_N + \sum_{j = N + 1}^\infty \q_j * \q_j * f\Big) = \lim_{M \rightarrow \infty} \phi_M * f  - \lim_{N \rightarrow - \infty} \big(\phi_N * f  - p_N \big)= f $$
as required.
\end{proof}

A similar argument gives Theorem \ref{thm - classical calderon on B}.

\begin{proof}[Proof of Theorem \ref{thm - classical calderon on B}]
 Let $\phi$ be as in the proof of Theorem \ref{thm - classical calderon on s'}. As previously, the key point is to study the convergence of $\phi_N*f$ as $N \rightarrow  - \infty$. Define the polynomials $p_N(x)$ as
            $$ p_N(x) = \sum_{|\gamma|\les [\alpha]} \frac{x^\gamma}{\gamma!} \p^\gamma\big(\phi_N*f\big)(0).$$
The form of the Taylor series remainder given in \cite{Folland1990} implies that for any $N<N'$
    \begin{align*}
      \big| \big(\phi_N*f - &p_N\big)(x) - \big(\phi_{N'} *f - p_{N'}\big)(x) \big| \\
      &\les \sum_{j=N+1}^{N'} \Big| \q_j * \q_j * f(x) - \sum_{|\gamma| \les [\alpha]} \frac{x^\gamma}{\gamma!} \p^\gamma(\q_j*\q_j*f)(0) \Big| \\
       &\lesa |x|^{[\alpha]} \sum_{j=N+1}^{N'} \sum_{|\gamma| = [\alpha]} \int_0^1 \big|  \p^\gamma(\q_j*\q_j*f)(tx) - \p^\gamma(\q_j*\q_j*f)(0)\big| dt.
    \end{align*}
If we now observe that
\begin{align*}
  \big|  \p^\gamma(\q_j*\q_j*f)(tx) - &\p^\gamma(\q_j*\q_j*f)(0)\big|\\
  &\lesa 2^{ j |\gamma|} \sup_{0<t<1} \big\| (\p^\gamma \q)( t 2^j x + \cdot) - (\p^\gamma \q)(\cdot) \big\|_{L^1}  \| \q_j * f \|_{L^\infty} \\
  &\lesa \| f \|_{\dot{B}^\alpha_{\infty, \infty}} 2^{j( |\gamma| - \alpha)} \min\{ 1, |x| 2^j \}
\end{align*}
we obtain the inequality
    \begin{equation}\label{eqn - thm calderon with besov - growth bound on phi*f}
    \big| \big(\phi_N*f - p_N\big)(x) - \big(\phi_{N'} *f - p_{N'}\big)(x) \big| \lesa |x|^{[\alpha]} \sum_{j=N+1}^{N'} 2^{j([\alpha]-\alpha)} \min\{ 1, 2^j |x| \}.
    \end{equation}
In particular, we have
    $$ \big| \big(\phi_N*f - p_N\big)(x) - \big(\phi_{N'} *f - p_{N'}\big)(x) \big| \lesa |x|^{[\alpha]+1} \sum_{j \les N'} 2^{ j([\alpha]+1-\alpha)} \lesa |x|^{[\alpha]+1} 2^{ N'([\alpha]+1 - \alpha)}.$$
Consequently $\phi_N*f - p_N$ forms a Cauchy sequence in $\s'$ as $N \rightarrow - \infty$ and hence converges to some $g \in \s'$. On the other hand it is easy to check that $\supp \widehat{ \phi_N *f} \subset \{ |\xi| \les 2^{N+1}\}$ and hence we must have $\supp \widehat{g} \subset \{ 0 \}$. Therefore $g = p$ for some polynomial $p$ and thus, from (\ref{eqn - thm cald on s' - phi ident}), we deduce that
    $$ \lim_{N\rightarrow - \infty} \bigg( p_N + \sum_{N+1}^\infty \q_j * \q_j * f\bigg)  = f  - \lim_{ N \rightarrow - \infty} \big( \phi_N*f - p_N \big) = f - p$$
as claimed.

It remains to prove the growth bound. To this end, let $N<0\les M$ and write
        $$ p_N + \sum_{j=N+1}^M \q_j*\q_j*f = \sum_{j=1}^{M} \q_j*\q_j*f + p_{N} + \sum_{j=N+1}^0 \q_j*\q_j*f. $$
To control the first term, we note that the support of $\q$ together with an application of Lemma \ref{lem - main dilation estimate} implies that $ \| \rho * \q_j \|_{L^1} \lesa 2^{ -j\beta}$ for every $\beta>0$. Thus choosing $\beta$ sufficiently large we have
        $$ \bigg|\rho*\bigg( \sum_{j=1}^M \q_j*\q_j*f\bigg)(x) \bigg|\les \sum_{j\g1} \| \rho*\q_j \|_{L^1} \| \q_j *f \|_{L^\infty} \lesa \sum_{j \g 1} 2^{ - j( \beta + \alpha)}\lesa 1.$$
On the other hand, for the second sum we note that if $\alpha<0$ then $p_N=0$ and we can simply write
    $$ \sum_{j=N+1}^0 \big| \q_j*\q_j*f(x) \big| \les \sum_{j \les 0} 2^{ - j \alpha} \| f \|_{\dot{B}^\alpha_{\infty, \infty} }\lesa 1$$
which gives the required estimate in the case $\alpha<0$. If $\alpha\g0$, we apply (\ref{eqn - thm cald on s' - phi ident}) followed by (\ref{eqn - thm calderon with besov - growth bound on phi*f}) to deduce that
        \begin{align*} \big| p_{N}(x) + \sum_{j=N+1}^0 &\q_j*\q_j*f(x)\big| \\
            &\les \big|p_0(x)\big| + \big| \phi_0*f(x) - p_0(x) - \big(\phi_N*f(x) - p_N(x) \big)\big| \\
        &\lesa ( 1 + |x|)^{[\alpha]} + |x|^{[\alpha]+1}  \sum_{ j \les - \log_2(|x|) } 2^{  j ([\alpha]+1- \alpha)  }  + |x|^{[\alpha]} \sum_{ - \log_2(|x|)  \les j \les 0} 2^{  j([\alpha]- \alpha)} \\
      &\lesa 1 + |x|^{[\alpha]+1} 2^{ -  ( [\alpha]+1 - \alpha) \log_2(|x|) }  + |x|^{[\alpha]} \begin{cases}  \log_2(|x|) \qquad &\alpha \in \NN \\
                                                                        2^{ -([\alpha]-\alpha) \log_2(|x|) } &\alpha \not \in \NN \end{cases}\\
      &\lesa 1 + \begin{cases}
        |x|^\alpha \log{ |x| } \qquad &\alpha \in \NN \\
        |x|^\alpha      &\alpha \not \in \NN.
      \end{cases}
       \end{align*}
As before, taking the convolution with $\rho$ then gives the required estimate. Thus the result follows.
\end{proof}

\providecommand{\bysame}{\leavevmode\hbox to3em{\hrulefill}\thinspace}
\providecommand{\MR}{\relax\ifhmode\unskip\space\fi MR }
\providecommand{\MRhref}[2]{%
  \href{http://www.ams.org/mathscinet-getitem?mr=#1}{#2}
}
\providecommand{\href}[2]{#2}


\begin{thebibliography}{10}

\bibitem{Berg1976}
J.~Bergh and J.~L{\"o}fstr{\"o}m,  \emph{Interpolation spaces. {A}n introduction}, Springer-Verlag, Berlin, 1976, Grundlehren der
Mathematischen Wissenschaften, no.~223.

\bibitem{Bui1979}
H.-Q. Bui, \emph{Harmonic functions, {R}iesz potentials, and the {L}ipschitz
  spaces of {H}erz}, Hiroshima Math. J. \textbf{9} (1979), no.~1, 245--295.


\bibitem{Bui1984}
H.-Q. Bui, \emph{Characterizations of weighted {B}esov and
  {T}riebel-{L}izorkin spaces via temperatures}, J. Funct. Anal. \textbf{55}
  (1984), no.~1, 39--62.

\bibitem{Bui2008}
H.-Q. Bui and T.~Candy, \emph{Characterisations of function spaces via
  non-smooth kernels}, talk given at the 7th Australia-New Zealand Mathematics
  Convention, Christchurch (8-12 Dec 2008).

\bibitem{Bui1996}
H.-Q. Bui, M.~Paluszy\'{n}ski, and M.H.~Taibleson, \emph{A maximal function
  characterization of weighted {B}esov-{L}ipschitz and {T}riebel-{L}izorkin
  spaces}, Studia Math. \textbf{119} (1996), 219--246.

\bibitem{Bui1997}
\bysame, \emph{Characterization of the
  {B}esov-{L}ipschitz and {T}riebel-{L}izorkin spaces, the case $q < 1$},
  J.  Fourier Anal.  Appl. \textbf{3} (1997), 837--846.

\bibitem{Bui2000}
H.-Q. Bui and M.H.~Taibleson, \emph{The characterization of the
  {T}riebel-{L}izorkin spaces for {$p=\infty$}}, J. Fourier Anal. Appl.
  \textbf{6} (2000), no.~5, 537--550.


\bibitem{Candy2008}
T.~Candy, \emph{A study of {B}esov-{L}ipschitz and {T}riebel-{L}izorkin
  spaces using non-smooth kernels}, Master's thesis, University of Canterbury,
  2008.

\bibitem{Fefferman1971}
C.~Fefferman and E.M.~Stein, \emph{Some maximal inequalities}, Amer. J. Math.
  \textbf{93} (1971), 107--115.

\bibitem{FS1972}
\bysame, \emph{${H}^p$ spaces of severable variables}, Acta
  Math. \textbf{129} (1972), 137--188.

\bibitem{Folland1990}
G.B. Folland, \emph{Remainder estimates in {T}aylor's theorem}, Amer. Math.
  Monthly \textbf{97} (1990), no.~3, 233--235.

\bibitem{Frazier1990}
M.~Frazier and B.~Jawerth, \emph{A discrete transform and
  decompositions of distribution spaces}, J. Funct. Anal. \textbf{93} (1990),
  no.~1, 34--170.

\bibitem{Heideman1974}
N.J.H.~Heideman, \emph{Duality and fractional integration in {L}ipschitz
  spaces}, Studia Math. \textbf{50} (1974), 65--85.

\bibitem{Herz1968}
C.S.~Herz, \emph{Lipschitz spaces and {B}ernstein's theorem on absolutely
  convergent {F}ourier transforms}, J. Math. Mech. \textbf{18} (1968/69),
  283--323.

\bibitem{Janson1981}
S.~Janson and M.H.~Taibleson, \emph{{I teoremi di rappresentazione di
  {C}alder\'{o}n}}, Rend. Sem. Mat. Univ. Politec. Torino \textbf{39} (1981),
  27-- 35.


\bibitem{Peetre1975}
J.~Peetre, \emph{On spaces of {T}riebel-{L}izorkin type}, Ark. Mat.
  \textbf{13} (1975), 123--130.

\bibitem{Peetre1976}
\bysame, \emph{{New thoughts on {B}esov Spaces}}, Duke University Mathematics
  Series, Duke University, 1976.


\bibitem{Rychkov1999a}
V.S.~Rychkov, \emph{On a theorem of {B}ui, {P}aluszy\'nski, and {T}aibleson},
  Tr. Mat. Inst. Steklova \textbf{227} (1999), no.~Issled. po Teor. Differ.
  Funkts. Mnogikh Perem. i ee Prilozh. 18, 286--298.

\bibitem{Stein1993}
E.M.~Stein, \emph{Harmonic analysis: real-variable methods, orthogonality,
  and oscillatory integrals}, Princeton Mathematical Series, vol.~43, Princeton
  University Press, Princeton, NJ, 1993, With the assistance of Timothy S.
  Murphy, Monographs in Harmonic Analysis, III.

\bibitem{Stromberg1989}
J-O.~Str{\"o}mberg and A.~Torchinsky, \emph{Weighted {H}ardy spaces},
Lecture Notes in Mathematics, vol. 1381, Springer-Verlag, Berlin, 1989.

\bibitem{Triebel1988}
H.~Triebel, \emph{Characterizations of {Besov-Hardy-Sobolev Spaces}: A
  unified approach}, J.  Approx. Theory \textbf{52} (1988),
  162--203.

\bibitem{Triebel1992}
\bysame, \emph{Theory of function spaces. {II}}, Monographs in Mathematics,
  vol.~84, Birkh\"auser Verlag, Basel, 1992.

\bibitem{Triebel2015}
\bysame, \emph{Tempered Homogeneous Function Spaces}, preprint, 2015.


\bibitem{vanEssen2005}
A.~van Essen, \emph{{A study of the {C}alder\'{o}n representation formula and
  its applications}}, Master's thesis, University of Canterbury, 2005.

\bibitem{Yosida1995}
K.~Yosida, \emph{Functional analysis}, Classics in Mathematics,
  Springer-Verlag, Berlin, 1995, Reprint of the sixth (1980) edition.

\end{thebibliography}
\end{document}